\documentclass[11pt,a4paper]{article}
\usepackage{epsf,epsfig,amsfonts,amsgen,amsmath,amstext,amsbsy,amsopn,amsthm,lineno}
\usepackage{color}
\usepackage{tikz}
\usepackage{cite}
\newtheorem{theorem}{Theorem}

\newtheorem{lemma}{Lemma}
\newtheorem{false statement}{False statement}

\theoremstyle{definition}

\newtheorem{claim}{Claim}

\newtheorem{conjecture}{Conjecture}

\newtheorem{case}{Case}

\newcounter{mathitem}
  {\begin{list}{{$(\roman{mathitem})$}}{
   \setcounter{mathitem}{0}
   \usecounter{mathitem}
   \setlength{\topsep}{0pt plus 2pt minus 0pt}
   \setlength{\parskip}{0pt plus 2pt minus 0pt}
   \setlength{\partopsep}{0pt plus 2pt minus 0pt}
   \setlength{\parsep}{0pt plus 2pt minus 0pt}
   \setlength{\leftmargin}{25pt}
   \setlength{\itemsep}{0pt plus 2pt minus 0pt}}}
  {\end{list}}

\begin{document}

\title{\bf\Large Anti-Ramsey numbers for vertex-disjoint triangles\thanks{Supported by NSFC (No. 12071370, 12131013, 12171393 and U1803263) and Natural Science Foundation of Shaanxi Provincial Department of Education (2021JM-040).}}
\date{}

\author{Fangfang Wu$^{a,b}$,
Shenggui Zhang$^{a,b,}$\thanks{Corresponding author.
E-mail addresses:
wufangfang2017@mail.nwpu.edu.cn (F. Wu),
sgzhang@nwpu.edu.cn (S. Zhang),
binlongli@nwpu.edu.cn (B. Li), xiaojimeng@mail.nwpu.edu.cn (J. Xiao).},
Binlong Li$^{a,b}$,
Jimeng Xiao$^{a,b}$\\[2mm]
\small $^{a}$School of Mathematics and Statistics, \\
\small Northwestern Polytechnical University, Xi'an, Shaanxi 710129, China\\
\small $^{b}$Xi'an-Budapest Joint Research Center for Combinatorics, \\
\small Northwestern Polytechnical University, Xi'an, Shaanxi 710129, China
}
\maketitle

\begin{abstract}
An edge-colored graph is called rainbow if all the colors on its edges are distinct. Given a positive integer $n$ and a graph $G$, the anti-Ramsey number $ar(n, G)$ is the maximum number of colors in an edge-coloring of $K_{n}$ with no rainbow copy of $G$. Denote by $kC_{3}$ the union of $k$ vertex-disjoint copies of $C_{3}$. In this paper, we determine the anti-Ramsey number $ar(n, kC_{3})$ for $n=3k$ and $n\geq2k^{2}-k+2$, respectively. When $3k\leq n\leq 2k^{2}-k+2$, we give lower and upper bounds for $ar(n, kC_{3})$.

\smallskip
\end{abstract}
\section{Introduction}\label{se1}
All graphs considered in this paper are finite and simple. For terminology and notations not defined here, we refer the reader to \cite{Bondy_Murty}.

Let $G$ be a graph. We denote by $V(G)$ and $E(G)$ the vertex set and the edge set of $G$, respectively, and use $v(G)$ and $e(G)$ to denote the numbers of vertices and edges in $G$. For a vertex $v$ of $G$ and a subgraph $H$ of $G$, $N_{H}(v)$ is the set of neighbors of $v$ contained in $H$. We let $d_{H}(v)=|N_{H}(v)|$. Thus $d_{G}(v)$ is the {\em degree} of $v$ in $G$. Denote by $\delta(G)$ and $\Delta(G)$ the {\em minimum and maximum degrees} of the vertices of $G$. For two vertex-disjoint subgraphs $F$ and $H$ of $G$, we use $E(F, H)$ to denote the set of edges in $G$ between $F$ and $H$, and $e(F, H)$ the number of edges between $F$ and $H$. If $F$ contains only one vertex $f$, we write $E(f, H)$ for $E(\{f\}, H)$ and $e(f, H)$ for $e(\{f\}, H)$. For a nonempty subset $V'\subseteq V(G)$ (or $E'\subseteq E(G)$), we use $G[V']$ (or $G[E']$) to denote the subgraph of $G$ {\em induced} by $V'$ (or $E'$), and $G-V'$ (or $G-E'$) to denote the subgraph $G[V(G)\setminus V']$ (or $G[E(G)\setminus E']$). When $V'=\{v\}$ (or $E'=\{e\}$), we use $G-v$ (or $G-e$) instead of $G-\{v\}$ (or $G-\{e\}$). Furthermore, for a subgraph $G'$ of $G$, we write $G-G'$ for $G-V(G')$.

An {\em edge-coloring} of $G$ is a mapping $C: E(G)\rightarrow \mathbb{N}$, where $\mathbb{N}$ is the natural number set. We call $G$ an {\em edge-colored graph} if it is assigned such an edge-coloring $C$ and use $C(G)$ to denote the set, and $c(G)$ the number of colors of edges in $G$. An edge-colored graph is called {\em rainbow} if all the colors on its edges are distinct. Given a positive integer $n$ and a graph $G$, the {\em anti-Ramsey number} $ar(n, G)$ is the maximum number of colors in an edge-coloring of $K_{n}$ with no rainbow copy of $G$ and the {\em Tur\'{a}n number} $ex(n, G)$ is the maximum number of edges of a graph on $n$ vertices containing no subgraph isomorphic to $G$. By taking one edge of each color in an edge-coloring of $K_{n}$, one can show that $ar(n, G)\leq ex(n, G)$.

Anti-Ramsey numbers were introduced by Erd\H{o}s, Simonovits and S\'{o}s \cite{ESS}. They showed that these are closely related to Tur\'{a}n numbers. Since then numerous results were established for a variety of graphs, including, among others, cycles \cite{Alon, ESS, JW, JSW, JL, MN}, cliques \cite{ESS, Sch, MN2}, paths \cite{SS}, matchings \cite{CLT, FKSS, HY, Sch} and trees \cite{Jiang, JW2}. If the readers want to know other related results, refer a survey paper \cite{FMO}.

An interesting open problem concerning anti-Ramsey numbers is the determination of the anti-Ramsey number of cycles. Erd\H{o}s, Simonovits and S\'{o}s \cite{ESS} conjectured that $ar(n, C_{l})=(\frac{l-2}{2}+\frac{1}{l-1})n+O(1)$ and proved it for $l=3$. Alon \cite{Alon} proved this conjecture for $l=4$ by showing that $ar(n, C_{4})=\big\lfloor\frac{4n}{3}\big\rfloor-1$. Jiang, Schiermeyer and West \cite{JSW} proved the conjecture for $l\leq7$. Finally, Montellano-Ballesteros and Neumann-Lara \cite{MN} completely proved this conjecture. Denote by $\Omega_{k}$ the union of $k$ vertex-disjoint cycles. Jin and Li \cite{JL} determined the anti-Ramsey number $ar(n, \Omega_{2})$ and gave an upper bound of $ar(n, \Omega_{k})$.

Let $kC_{3}$ denote the union of $k$ vertex-disjoint copies of $C_{3}$. Yuan and Zhang \cite{Yuan} provide the exact results of $ar(n, kC_{3})$ when $n$ is sufficiently large. In this paper, we improve the result of Yuan and Zhang \cite{Yuan} from $n$ sufficiently large to $n\geq2k^{2}-k+2$, and give some bounds and exact values for the other $n$. Note that the result of Erd\H{o}s, Simonovits and S\'{o}s \cite{ESS} implies $ar(n, C_{3})=n-1$. Therefore, we assume $k\geq2$ in the rest of the paper. Obviously, if $n<3k$, then $K_{n}$ does not have enough vertices for $k$ vertex-disjoint triangles, which implies that a rainbow $K_{n}$ does not have $kC_{3}$'s. Hence $ar(n, kC_{3})=\binom{n}{2}$ when $n<3k$. For $n\geq3k$, the question becomes nontrivial. We first give a lower bound for $ar(n, kC_{3})$.

\begin{theorem}\label{lb}
$ar(n, kC_{3})\geq\max\big\{\binom{3k-1}{2}+n-3k+1, \left\lfloor\frac{(n-k+2)^{2}}{4}\right\rfloor+(k-2)(n-k+2)+\binom{k-2}{2}+1\big\}$ for all $n\geq3k$.
\end{theorem}

If $n=3k$, then the first number is greater than the second number in the formula of Theorem \ref{lb}. The following theorem shows that $ar(3k, kC_{3})$ is equal to the first number in the formula of Theorem \ref{lb}, which just attains the lower bound of Theorem \ref{lb}.

\begin{theorem}\label{5}
$ar(3k, kC_{3})=\binom{3k-1}{2}+1$.
\end{theorem}

Our next theorem gives an upper bound of $ar(n, kC_{3})$ with $3k\leq n\leq 2k^{2}-k+2$.

\begin{theorem}\label{13}
$ar(n, kC_{3})\leq\left\lfloor\frac{(n-k+2)^{2}}{4}\right\rfloor+(k-2)(n-k+2)+\binom{k-2}{2}+(k-1)^{2}-\frac{n-3k}{2}+1$ for all $3k\leq n\leq2k^{2}-k+2$.
\end{theorem}

When $n=3k$, the upper bound in Theorem \ref{13} is equal to $\binom{3k-1}{2}+1$. Theorem \ref{5} shows that $ar(3k, kC_{3})=\binom{3k-1}{2}+1$ attains the upper bound of Theorem \ref{13}. The anti-Ramsey number $ar(n, kC_{3})$ for $n\geq2k^{2}-k+2$ is determined in the following theorem, which matches the lower bound in Theorem \ref{lb}.

\begin{theorem}\label{12}
$ar(n, kC_{3})=\left\lfloor\frac{(n-k+2)^{2}}{4}\right\rfloor+(k-2)(n-k+2)+\binom{k-2}{2}+1$ for all $n\geq2k^{2}-k+2$.
\end{theorem}

From the above conclusions, we believe that equality holds in the inequality of Theorem \ref{lb} and propose the following conjecture.

\begin{conjecture}\label{Con1}
$ar(n, kC_{3})=\max\big\{\binom{3k-1}{2}+n-3k+1, \left\lfloor\frac{(n-k+2)^{2}}{4}\right\rfloor+(k-2)(n-k+2)+\binom{k-2}{2}+1\big\}$ for all $n\geq3k$.
\end{conjecture}

The next section will be focused on introducing some lemmas, which are needed tools to derive our main results. Some additional definitions and notations will be introduced in the next section as well. The proofs of our theorems will be given in Section \ref{se3}.

\section{Preliminaries}\label{se2}
\begin{lemma}\label{3k}
Let $k\geq4$ be a positive integer. If $G$ is a graph on $3k$ vertices such that $e(G)\geq\binom{3k-1}{2}+2$, then $G$ contains a $kC_{3}$.
\end{lemma}

The {\em complement} of a graph $G$, denote by $\overline{G}$, is the graph whose vertex set is $V(G)$ and whose edges are the pairs of nonadjacent vertices of $G$. By considering complements, we obtain the following equivalent formulation of Lemma \ref{3k}.

\begin{lemma}\label{3k'}
Let $k\geq4$ be a positive integer. If $G$ is a graph on $3k$ vertices such that $e(G)\leq3k-3$, then $G$ is $k$-partite with parts of size $3$.
\end{lemma}

In the proof of Lemma \ref{3k'} we shall make use of the following result.

\begin{theorem}[Hajnal and Szemer\'{e}di \cite{Hajnal_Szemeredi}]\label{HS}
Let $k$ be a positive integer and $G$ a graph on $3k$ vertices with $\Delta(G)\leq k-1$. Then $G$ is $k$-partite with parts of size $3$.
\end{theorem}

\noindent\textbf{Proof of Lemma \ref{3k'}.}
We prove the lemma by induction on $k$ and start with the following claim.

\begin{claim}\label{c1}
If $v(G)=12$ and $e(G)\leq9$, then $G$ is $4$-partite with parts of size 3.
\end{claim}

\begin{proof}
By Theorem \ref{HS}, we can assume that there exists a vertex $u\in V(G)$ such that $d_{G}(u)\geq4$. Since $v(G)=12$ and $e(G)\leq9$, the number of components of $G$ is at least 3. Thus there is an independent set $A$ in $G$ such that $u\in A$ and $|A|=3$. Let $G'=G-A$. Since $d_{G}(u)\geq4$ and $e(G)\leq9$, we have $e(G')\leq5$. If $G'$ is $3$-partite with parts of size 3, then $G$ is $4$-partite with parts of size 3. So by Theorem \ref{HS}, we assume that there exists a vertex $w\in V(G')$ such that $d_{G'}(w)\geq3$. Since $v(G')=9$ and $e(G')\leq5$, the number of components of $G'$ is at least 4. Thus there is an independent set $B$ in $G'$ such that $w\in B$ and $|B|=3$. Let $G''=G'-B$. Since $d_{G'}(w)\geq3$ and $e(G')\leq5$, we have $e(G'')\leq2$. Clearly, $G''$ is $2$-partite with parts of size 3, which implies that $G'$ is $3$-partite with parts of size 3. Further, we have $G$ is $4$-partite with parts of size 3.
\end{proof}

The base case for $k=4$ can be obtained from Claim \ref{c1}. By Theorem \ref{HS}, we may assume that there exists a vertex $u\in V(G)$ such that $d_{G}(u)\geq k$. Since $v(G)=3k$ and $e(G)\leq3k-3$, the number of components of $G$ is at least $3$. Thus there is an independent set $A$ in $G$ such that $u\in A$ and $|A|=3$. Let $G'=G-A$. Since $d_{G}(u)\geq k$ and $e(G)\leq3k-3$, we have $e(G')\leq3k-3-k<3k-3-3=3(k-1)-3$. Thus, by the induction hypothesis, we see that $G'$ is $(k-1)$-partite with parts of size $3$ which implies that $G$ is $k$-partite with parts of size $3$. {\hfill$\Box$}

The following two lemmas are technical lemmas that deal with the construction of vertex-disjoint triangles in a given graph.

\begin{lemma}[Wang \cite{Wang}]\label{W1}
Let $P=uvw$ be a path and $C$ a triangle in $G$ such that $C$ is vertex-disjoint from $P$. Then the following hold:

({\romannumeral1})~If $e(u, C)+e(w, C)\geq5$ and $e(v, C)\geq1$, then $G[V(P)\cup V(C)]$ contains two vertex-disjoint triangles;

({\romannumeral2})~If $e(P, C)\geq7$, then either $G[V(P)\cup V(C)]$ contains two vertex-disjoint triangles, or $e(u, C)=e(w, C)=2$, $N_{C}(u)=N_{C}(w)$ and $e(v, C)=3$.
\end{lemma}

\begin{lemma}\label{l4}
Let $M$ be a matching of size $2$ and $C$ a triangle in $G$ such that $C$ is vertex-disjoint from $M$. If $e(M, C)\geq9$, then $G[V(M)\cup V(C)]$ contains two vertex-disjoint triangles.
\end{lemma}

\noindent\textbf{Proof of Lemma \ref{l4}.}
Let $M=\{ab, cd\}$ and $C=v_{1}v_{2}v_{3}v_{1}$. Since $e(M, C)\geq9$, it may be assumed without loss of generality that one of the alternatives $e(ab, C)=5$ and $e(ab, C)=6$ holds. If $e(ab, C)=5$, then $e(cd, C)\geq4$. Without loss of generality, assume that $a$ is joined to $v_{1}, v_{2}, v_{3}$ and $b$ to $v_{1}, v_{2}$. Furthermore $e(v_{i}, cd)=2$ for at least one $i$, $1\leq i\leq 3$. Then $cdv_{i}c$ and one of $abv_{1}a$, $abv_{2}a$ together constitute two vertex-disjoint triangles in $G[V(M)\cup V(C)]$. If $e(ab, C)=6$, then $a$ and $b$ are joined to $v_{1}, v_{2}, v_{3}$, also $e(cd, C)\geq3$, so it may be assumed without loss of generality that $cv_{1}, cv_{2}\in E(G)$. Then $cv_{1}v_{2}c$ and $abv_{3}a$ are two vertex-disjoint triangles in $G[V(M)\cup V(C)]$. This proves the lemma. {\hfill$\Box$}

In 1963, Dirac \cite{Dirac} gave the minimum degree condition for the existence of vertex-disjoint triangles.

\begin{theorem}[Dirac \cite{Dirac}]\label{D}
For all positive integers $n$ and $k$ with $n\geq3k$, every graph on $n$ vertices with minimum degree at least $\frac{n+k}{2}$ contains $k$ vertex-disjoint triangles.
\end{theorem}

With more effort, we can prove the following two stronger structural results.

\begin{lemma}\label{11}
Let $G$ be a graph on $n$ vertices with $n\geq3k$. If $\delta(G)\geq\frac{n+k-1}{2}$ and $G$ contains no $kC_{3}$'s, then there exists a partition $V(G)=V_{1}\cup V_{2}\cup V_{3}$, so that $|V_{1}|=k-1$, $|V_{2}|=|V_{3}|=\frac{n-k+1}{2}$, any two vertices in different parts are adjacent, and there are no edges joining pairs of vertices in $V_{2}$ and so does $V_{3}$, unless $G$ is isomorphic to $G'$ in Figure $1$ when $n=10$.
\end{lemma}

\begin{figure}[h] \label{fig:1}
\begin{center}
   \begin{tikzpicture}[scale=1.4,auto,swap]
   \tikzstyle{blackvertex}=[circle,draw=black]
   \node [label=left: ,blackvertex,scale=0.3] (a1) at (-2,1) {};
   \node [label=left: ,blackvertex,scale=0.3] (a2) at (-1.5,1.8) {};
   \node [label=left: ,blackvertex,scale=0.3] (a3) at (-1,1) {};
   \node [label=right: ,blackvertex,scale=0.3] (a4) at (-2,-0.5) {};
   \node [label=right: ,blackvertex,scale=0.3] (a5) at (-1.5,0.3) {};
   \node [label=right: ,blackvertex,scale=0.3] (a6) at (-1,-0.5) {};
   \node [label=right: ,blackvertex,scale=0.3] (a7) at (0,1.3) {};
   \node [label=right: ,blackvertex,scale=0.3] (a8) at (1.2,1.3) {};
   \node [label=right: ,blackvertex,scale=0.3] (a9) at (1.2,0) {};
   \node [label=right: ,blackvertex,scale=0.3] (a10) at (0,0) {};
   \draw [black,thick] (a1)--(a2);
   \draw [black,thick] (a1)--(a3);
   \draw [black,thick] (a1)--(a5);
   \draw [black,thick] (a1) .. controls (-1.3,0.5) .. (a6);
   \draw [black,thick] (a1) .. controls (-1.1,1.2) .. (a7);
   \draw [black,thick] (a1)--(a9);
   \draw [black,thick] (a2)--(a3);
   \draw [black,thick] (a2)--(a4);
   \draw [black,thick] (a2)--(a7);
   \draw [black,thick] (a2)--(a8);
   \draw [black,thick] (a2) .. controls (-0.8,1.2) .. (a10);
   \draw [black,thick] (a3) .. controls (-1.65,0.5) .. (a4);
   \draw [black,thick] (a3)--(a7);
   \draw [black,thick] (a3) .. controls (0,1.1) .. (a8);
   \draw [black,thick] (a3)--(a10);
   \draw [black,thick] (a4)--(a5);
   \draw [black,thick] (a4)--(a6);
   \draw [black,thick] (a4)--(a7);
   \draw [black,thick] (a4) .. controls (-1.2,-0.3) .. (a9);
   \draw [black,thick] (a5)--(a6);
   \draw [black,thick] (a5)--(a8);
   \draw [black,thick] (a5) .. controls (0,0.2) .. (a9);
   \draw [black,thick] (a5)--(a10);
   \draw [black,thick] (a6)--(a8);
   \draw [black,thick] (a6)--(a9);
   \draw [black,thick] (a6)--(a10);
   \draw [black,thick] (a7)--(a8);
   \draw [black,thick] (a7)--(a10);
   \draw [black,thick] (a8)--(a9);
   \draw [black,thick] (a9)--(a10);
\end{tikzpicture}

\small Figure 1. An exceptional graph $G'$ in Lemma \ref{11}.
\end{center}
\end{figure}

\begin{lemma}\label{111}
Let $G$ be a graph on $n$ vertices with $n\geq3k+1$. If $\delta(G)\geq\frac{n+k-2}{2}$ and $G$ contains a $(k-1)C_{3}$ but not $kC_{3}$'s, then there exists a subgraph $H\subseteq G$ such that $|V(H)|=3k-3$, $H$ contains a $(k-1)C_{3}$ and $G-H$ is a complete bipartite graph whose two parts are at least $2$ in size.
\end{lemma}

The proofs of Lemmas \ref{11} and \ref{111} are more complicated, and we will put them in the last section. Next, we introduce some additional definitions and notations.

Let $v$ be a vertex in an edge-colored graph $G$. A color $c\in C(G)$ is {\em saturated} by $v$ if all the edges with the color $c$ are incident to $v$. The {\em color saturated degree} of $v$, denote by $d^{s}(v)$, is the number of colors saturated by $v$. Clearly, $d^{s}(v)=c(G)-c(G-v)$. A {\em representing subgraph} in an edge-coloring of $K_{n}$ is a spanning subgraph containing exactly one edge of each color.

\section{Proofs of the theorems}\label{se3}
\textbf{Proof of Theorem \ref{lb}.}
We prove the theorem by giving two colorings of $K_{n}$. For the first coloring, we choose $K_{3k-1}$ and color all its edges with distinct colors. Only one color is added for each vertex in the remaining vertices added. This is an exact $\big(\binom{3k-1}{2}+n-3k+1\big)$-coloring with no rainbow $kC_{3}$'s (see Figure 2 (a)).

For the second coloring, we first choose $K_{k-2}\vee K_{\lfloor\frac{n-k+2}{2}\rfloor, \lceil\frac{n-k+2}{2}\rceil}$ and color all its edges with distinct colors, where $K_{k-2}\vee K_{\lfloor\frac{n-k+2}{2}\rfloor, \lceil\frac{n-k+2}{2}\rceil}$ is a graph obtained from the vertex-disjoint union of $K_{k-2}$ and $K_{\lfloor\frac{n-k+2}{2}\rfloor, \lceil\frac{n-k+2}{2}\rceil}$ by adding edges joining every vertex of $K_{k-2}$ to every vertex of $K_{\lfloor\frac{n-k+2}{2}\rfloor, \lceil\frac{n-k+2}{2}\rceil}$. Next, we put one extra color to all remaining edges in $K_{n}$. This is an exact $\big(\left\lfloor\frac{(n-k+2)^{2}}{4}\right\rfloor+(k-2)(n-k+2)+\binom{k-2}{2}+1\big)$-coloring with no rainbow $kC_{3}$'s (see Figure 2 (b)).

\begin{center}
\begin{tikzpicture}
\path (0cm,0cm) coordinate(O);
\path ++(-16cm,0.5cm) coordinate (P1) ++(1cm,0cm) coordinate (P2);
\foreach \i in {2}{
\draw (P\i) ellipse(0.75cm and 1.4cm);
\fill[gray!35] (P\i) ellipse (0.75cm and 1.4cm);
}
\tikzstyle{blackvertex}=[circle,draw=black]
	\node [label=below: $v_{1}$,blackvertex,scale=0.4] (a1) at (-13.25,0.3) {};
    \node [label=below: $v_{2}$,blackvertex,scale=0.4] (a2) at (-12.25,0.3) {};
    \node [label=below: $v_{n-3k}$,blackvertex,scale=0.4] (a3) at (-10.8,0.3) {};
    \node [label=below: $v_{n-3k+1}$,blackvertex,scale=0.4] (a4) at (-9.5,0.3) {};
    \node [label=below: $\textbf{\ldots}$] (a5) at (-11.5,0.6) {};
    \node [label=below: rb] (a6) at (-15,1.25) {};
    \node [label=below: $K_{3k-1}$] (a7) at (-15,0.75) {};
    \node [label=below: (a) The first coloring of $K_{n}$] (a8) at (-12.4,-2.3) {};
    \draw [red,line width=2.5pt] (a1)--(-14.25,0.3);
    \draw [blue,line width=2.5pt] (a2)--(a1);
    \draw [blue,line width=2.5pt] (a2) .. controls (-13.25,0.8) .. (-14.25,0.5);
    \draw [purple,line width=2.5pt] (a3) .. controls (-11.5,0.6) .. (a2);
    \draw [purple,line width=2.5pt] (a3) .. controls (-11.7,0.9) .. (a1);
	\draw [purple,line width=2.5pt] (a3) .. controls (-12.3,1.4) .. (-14.26,0.7);
    \draw [green,line width=2.5pt] (a4)--(a3);
    \draw [green,line width=2.5pt] (a4) .. controls (-11,1.1) .. (a2);
    \draw [green,line width=2.5pt] (a4) .. controls (-11.4,1.6) .. (a1);
    \draw [green,line width=2.5pt] (a4) .. controls (-11.5,2) .. (-14.27,0.9);

\path (0cm,0cm) coordinate(O);
\path ++(-8.8cm,2cm) coordinate (P1) ++(3cm,0cm) coordinate (P2);
\foreach \i in {2}{
\draw (P\i) ellipse(0.75cm and 1.4cm);
\fill[yellow!15] (P\i) ellipse (0.75cm and 1.4cm);
}
\tikzstyle{blackvertex}=[circle,draw=black]
	\node [label=left: $x_{1}$,blackvertex,scale=0.3] (b1) at (-5.65,3) {};
    \node [label=left: $x_{2}$,blackvertex,scale=0.3] (b2) at (-5.65,2.4) {};
    \node [label=left: $x_{s}$,blackvertex,scale=0.3] (b3) at (-5.65,1) {};
    \node [label=below: $\textbf{\vdots}$] (b4) at (-5.65,2.35) {};
\path (0cm,0cm) coordinate(O);
\path ++(-5.2cm,2cm) coordinate (P1) ++(3cm,0cm) coordinate (P2);
\foreach \i in {2}{
\draw (P\i) ellipse(0.75cm and 1.4cm);
\fill[yellow!15] (P\i) ellipse (0.75cm and 1.4cm);
}
\tikzstyle{blackvertex}=[circle,draw=black]
	\node [label=right: $y_{1}$,blackvertex,scale=0.3] (b5) at (-2.25,3) {};
    \node [label=right: $y_{2}$,blackvertex,scale=0.3] (b6) at (-2.25,2.4) {};
    \node [label=right: $y_{t}$,blackvertex,scale=0.3] (b7) at (-2.25,1) {};
    \node [label=below: $\textbf{\vdots}$] (b8) at (-2.25,2.35) {};
    \draw [red,thick] (b1)--(b5);
    \draw [blue,thick] (b1)--(b6);
    \draw [green,thick] (b1)--(b7);
    \draw [purple,thick] (b2)--(b5);
    \draw [brown,thick] (b2)--(b6);
    \draw [cyan,thick] (b2)--(b7);
    \draw [orange,thick] (b3)--(b5);
	\draw [pink,thick] (b3)--(b6);
    \draw [gray,thick] (b3)--(b7);
\node (1) at (-5.6,0.7) [label=below:{$s=\left\lfloor\frac{n-k+2}{2}\right\rfloor$}]{};
\node (2) at (-2.2,0.7) [label=below:{$t=\left\lceil\frac{n-k+2}{2}\right\rceil$}]{};

\path (0cm,0cm) coordinate(O);
\path ++(-7cm,-1.7cm) coordinate (P1) ++(3cm,0cm) coordinate (P2);
\foreach \i in {2}{
\draw (P\i) ellipse(1.3cm and 0.6cm);
\fill[gray!35] (P\i) ellipse (1.3cm and 0.6cm);
}
\tikzstyle{blackvertex}=[circle,draw=black]
    \node [label=below: rb $K_{k-2}$] (c1) at (-4,-1.25) {};
    \node [label=below: rb] (c2) at (-3.5,-0.2) {};
    \node [label=below: (b) The second coloring of $K_{n}$] (c3) at (-4,-2.3) {};
    \draw [black,thick] (-6.8,3.7)--(-1.1,3.7);
    \draw [black,thick] (-6.8,-0.2)--(-1.1,-0.2);
    \draw [black,thick] (-6.8,3.7)--(-6.8,-0.2);
    \draw [black,thick] (-1.1,3.7)--(-1.1,-0.2);
    \draw [red!50,line width=3pt] (-4.4,-1.13)--(-4.4,-0.2);
    \draw [orange!50,line width=3pt] (-4.3,-1.1)--(-4.3,-0.2);
    \draw [yellow!50,line width=3pt] (-4.2,-1.1)--(-4.2,-0.2);
    \draw [green!50,line width=3pt] (-4.1,-1.1)--(-4.1,-0.2);
    \draw [cyan!50,line width=3pt] (-4,-1.1)--(-4,-0.2);
    \draw [blue!50,line width=3pt] (-3.9,-1.1)--(-3.9,-0.2);
    \draw [purple!50,line width=3pt] (-3.8,-1.1)--(-3.8,-0.2);
\end{tikzpicture}

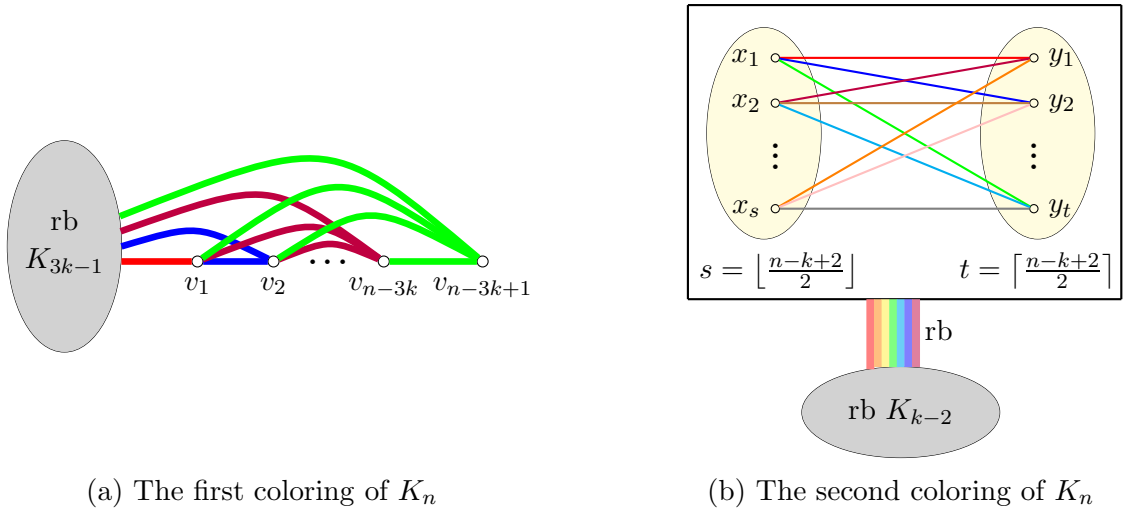
\small Figure 2. Two colorings of $K_{n}$ in the proof of Theorem \ref{lb}.
\end{center}

The proof is complete. {\hfill$\Box$}

\noindent\textbf{Proof of Theorem \ref{5}.}
The lower bound is due to Theorem \ref{lb}. In proving the upper bound, we color the edges of $K_{3k}$ with $\binom{3k-1}{2}+2$ colors and prove that a rainbow $kC_{3}$ can be found. Let $G$ be a representing subgraph. Then $e(G)=\binom{3k-1}{2}+2$, which implies that $e(\overline{G})=\binom{3k}{2}-\big(\binom{3k-1}{2}+2\big)=3k-3$. If $k\geq 4$, then by Lemma \ref{3k}, $G$ contains a $kC_{3}$ that form a rainbow $kC_{3}$ in $K_{3k}$. So we next consider the cases of $k=2$ and $k=3$.

\setcounter{case}{0}
\begin{case}
$k=2$.
\end{case}

In this case, we have $v(\overline{G})=6$ and $e(\overline{G})=3$. If $\Delta(\overline{G})=1$ or $\Delta(\overline{G})=2$ but $\overline{G}[E(\overline{G})]\neq C_{3}$ or $\Delta(\overline{G})=3$, then clearly $\overline{G}$ is a $2$-partite graph with parts of size $3$ which implies that $G$ contains a $2C_{3}$ that form a rainbow $2C_{3}$ in $K_{6}$. If $\overline{G}[E(\overline{G})]=C_{3}$, then $G$ is obtained from $K_{3,3}$ plus three edges in one class that form the rainbow spanning subgraph in $K_{6}$ (see Figure 3). Now $\{x_{1}y_{1}y_{2}x_{1}, y_{3}x_{2}x_{3}y_{3}\}$ or $\{x_{2}y_{1}y_{2}x_{2}, y_{3}x_{1}x_{3}y_{3}\}$ or $\{x_{3}y_{1}y_{2}x_{3}, y_{3}x_{1}x_{2}y_{3}\}$ is a rainbow $2C_{3}$ in $K_{6}$.

\begin{figure}[h] \label{fig:1}
\begin{center}
   \begin{tikzpicture}[scale=1.4,auto,swap]
   \tikzstyle{blackvertex}=[circle,draw=black]
   \node [label=left: $x_1$,blackvertex,scale=0.3] (a1) at (0,1) {};
   \node [label=left: $x_2$,blackvertex,scale=0.3] (a2) at (0,0) {};
   \node [label=left: $x_3$,blackvertex,scale=0.3] (a3) at (0,-1) {};
   \node [label=right: $y_1$,blackvertex,scale=0.3] (a4) at (2,1) {};
   \node [label=right: $y_2$,blackvertex,scale=0.3] (a5) at (2,0) {};
   \node [label=right: $y_3$,blackvertex,scale=0.3] (a6) at (2,-1) {};
   \node [label=below: (a) $G$] (a7) at (1,-1.1) {};
   \draw [black,thick] (a1)--(a4);
   \draw [black,thick] (a1)--(a5);
   \draw [black,thick] (a1)--(a6);
   \draw [black,thick] (a2)--(a4);
   \draw [black,thick] (a2)--(a5);
   \draw [black,thick] (a2)--(a6);
   \draw [black,thick] (a3)--(a4);
   \draw [black,thick] (a3)--(a5);
   \draw [black,thick] (a3)--(a6);
   \draw [black,thick] (a1)--(a2);
   \draw [black,thick] (a2)--(a3);
   \draw [black,thick] (a1) .. controls (-0.8,0) .. (a3);

   \node [label=left: $x_1$,blackvertex,scale=0.3] (a1') at (4.5,1) {};
   \node [label=left: $x_2$,blackvertex,scale=0.3] (a2') at (4.5,0) {};
   \node [label=left: $x_3$,blackvertex,scale=0.3] (a3') at (4.5,-1) {};
   \node [label=right: $y_1$,blackvertex,scale=0.3] (a4') at (6.5,1) {};
   \node [label=right: $y_2$,blackvertex,scale=0.3] (a5') at (6.5,0) {};
   \node [label=right: $y_3$,blackvertex,scale=0.3] (a6') at (6.5,-1) {};
   \node [label=below: (b) $K_{6}$] (a7') at (5.5,-1.1) {};
   \draw [blue,thick] (a1')--(a4');
   \draw [red,thick] (a1')--(a5');
   \draw [green,thick] (a1')--(a6');
   \draw [purple,thick] (a2')--(a4');
   \draw [orange,thick] (a2')--(a5');
   \draw [cyan,thick] (a2')--(a6');
   \draw [yellow,thick] (a3')--(a4');
   \draw [pink,thick] (a3')--(a5');
   \draw [gray,thick] (a3')--(a6');
   \draw [blue!40,thick] (a1')--(a2');
   \draw [green!40,thick] (a2')--(a3');
   \draw [orange!40,thick] (a1') .. controls (3.7,0) .. (a3');
   \draw [dotted,thick] (a4')--(a5');
   \draw [dotted,thick] (a5')--(a6');
   \draw [dotted,thick] (a4') .. controls (7.2,0) .. (a6');
\end{tikzpicture}

\small Figure 3. $G$ and $K_{6}$ in Case 1.
\end{center}
\end{figure}

\begin{case}
$k=3$.
\end{case}

In this case, we have $v(\overline{G})=9$ and $e(\overline{G})=6$. If $\Delta(\overline{G})\leq2$, then by Theorem \ref{HS}, $\overline{G}$ is a $3$-partite graph with parts of size $3$. This implies that $G$ contains a $3C_{3}$ that form a rainbow $3C_{3}$ in $K_{9}$. So we assume that there exists a vertex, say $u$, such that $d_{\overline{G}}(u)\geq3$. Since $\overline{G}$ has $9$ vertices and $6$ edges, the number of components of $\overline{G}$ is at least $3$. We know that there is an independent set $A$ of vertices such that $u\in A$ and $|A|=3$. Let $\overline{G}'$ be the graph obtained from $\overline{G}$ by removing the vertices in $A$. If $d_{\overline{G}}(u)\geq4$, then $e(\overline{G}')\leq6-4=2$, which implies that $\overline{G}'$ is a $2$-partite graph with parts of size $3$. Thus $\overline{G}$ is a $3$-partite graph with parts of size $3$. Clearly, we can find a $3C_{3}$ in $G$ that form a rainbow $3C_{3}$ in $K_{9}$. So we assume that $d_{\overline{G}}(u)=3$. Then $e(\overline{G}')\leq6-3=3$. Similar to the Case 1, when $\overline{G}'[E(\overline{G}')]\neq C_{3}$, we know that $\overline{G}'$ is a $2$-partite graph with parts of size $3$. Thus $\overline{G}$ is a $3$-partite graph with parts of size $3$, which implies that $G$ contains a $3C_{3}$ that form a rainbow $3C_{3}$ in $K_{9}$. If $\overline{G}'[E(\overline{G}')]=C_{3}$, according to the value of $|N_{\overline{G}'}(u)\cap V(\overline{G}'[E(\overline{G}')])|$, we distinguish the following four cases.

(1)~$|N_{\overline{G}'}(u)\cap V(\overline{G}'[E(\overline{G}')])|=0$, see Figure 4 (a);

(2)~$|N_{\overline{G}'}(u)\cap V(\overline{G}'[E(\overline{G}')])|=1$, see Figure 4 (b);

(3)~$|N_{\overline{G}'}(u)\cap V(\overline{G}'[E(\overline{G}')])|=2$, see Figure 4 (c);

(4)~$|N_{\overline{G}'}(u)\cap V(\overline{G}'[E(\overline{G}')])|=3$, see Figure 4 (d).

We use the solid lines to denote the edges of $G$ and the dotted lines to denote the edges of $\overline{G}$ in Figure 3. Clearly, in Cases (1), (2) and (3), we can find a $3C_{3}$ in $G$ that form a rainbow $3C_{3}$ in $K_{9}$. The graph $G$ in Case (4) form the rainbow spanning subgraph in the corresponding $K_{9}$ and we can still find a rainbow $3C_{3}$ in this $K_{9}$ using $\{u'uvu', v'u''v''v', ww'w''w\}$ or $\{u''uvu'', v'u'w''v', wv''w'w\}$ or $\{w'uvw', v'v''w''v', wu'u''w\}$.

\begin{figure}[h] \label{fig:1}
\begin{center}
   \begin{tikzpicture}[scale=1.4,auto,swap]
   \tikzstyle{blackvertex}=[circle,draw=black]
   \node [label=left: $u$,blackvertex,scale=0.3] (a1) at (0,1) {};
   \node [label=left: $u'$,blackvertex,scale=0.3] (a2) at (0,0) {};
   \node [label=left: $u''$,blackvertex,scale=0.3] (a3) at (0,-1) {};
   \node [label=right: $v$,blackvertex,scale=0.3] (a4) at (1,1) {};
   \node [label=right: $v'$,blackvertex,scale=0.3] (a5) at (1,0) {};
   \node [label=right: $v''$,blackvertex,scale=0.3] (a6) at (1,-1) {};
   \node [label=right: $w$,blackvertex,scale=0.3] (a7) at (2,1) {};
   \node [label=right: $w'$,blackvertex,scale=0.3] (a8) at (2,0) {};
   \node [label=right: $w''$,blackvertex,scale=0.3] (a9) at (2,-1) {};
   \node [label=below: (a) The graph in Case 2 (1)] (a10) at (1,-1.1) {};
   \draw [dotted,thick] (a1)--(a6);
   \draw [dotted,thick] (a1)--(a8);
   \draw [dotted,thick] (a1) .. controls (1,-0.4) .. (a9);
   \draw [dashdotted,thick] (a4)--(a5);
   \draw [dashdotted,thick] (a4)--(a7);
   \draw [dashdotted,thick] (a5)--(a7);
   \draw [black,thick] (a1)--(a2);
   \draw [black,thick] (a1)--(a4);
   \draw [black,thick] (a2)--(a4);
   \draw [black,thick] (a3)--(a5);
   \draw [black,thick] (a3)--(a6);
   \draw [black,thick] (a5)--(a6);
   \draw [black,thick] (a7)--(a8);
   \draw [black,thick] (a8)--(a9);
   \draw [black,thick] (a7) .. controls (2.8,0) .. (a9);

   \node [label=left: $u$,blackvertex,scale=0.3] (a1) at (4.5,1) {};
   \node [label=left: $u'$,blackvertex,scale=0.3] (a2) at (4.5,0) {};
   \node [label=left: $u''$,blackvertex,scale=0.3] (a3) at (4.5,-1) {};
   \node [label=right: $v$,blackvertex,scale=0.3] (a4) at (5.5,1) {};
   \node [label=right: $v'$,blackvertex,scale=0.3] (a5) at (5.5,0) {};
   \node [label=right: $v''$,blackvertex,scale=0.3] (a6) at (5.5,-1) {};
   \node [label=right: $w$,blackvertex,scale=0.3] (a7) at (6.5,1) {};
   \node [label=right: $w'$,blackvertex,scale=0.3] (a8) at (6.5,0) {};
   \node [label=right: $w''$,blackvertex,scale=0.3] (a9) at (6.5,-1) {};
   \node [label=below: (b) The graph in Case 2 (2)] (a10) at (5.5,-1.1) {};
   \draw [dotted,thick] (a1)--(a4);
   \draw [dotted,thick] (a1)--(a6);
   \draw [dotted,thick] (a1)--(a8);
   \draw [dashdotted,thick] (a4)--(a5);
   \draw [dashdotted,thick] (a4)--(a7);
   \draw [dashdotted,thick] (a5)--(a7);
   \draw [black,thick] (a1)--(a2);
   \draw [black,thick] (a2)--(a7);
   \draw [black,thick] (a1) .. controls (5.5,1.4) .. (a7);
   \draw [black,thick] (a3)--(a5);
   \draw [black,thick] (a3)--(a6);
   \draw [black,thick] (a5)--(a6);
   \draw [black,thick] (a4)--(a8);
   \draw [black,thick] (a4)--(a9);
   \draw [black,thick] (a8)--(a9);

   \node [label=left: $u$,blackvertex,scale=0.3] (a1) at (0,-2.5) {};
   \node [label=left: $u'$,blackvertex,scale=0.3] (a2) at (0,-3.5) {};
   \node [label=left: $u''$,blackvertex,scale=0.3] (a3) at (0,-4.5) {};
   \node [label=right: $v$,blackvertex,scale=0.3] (a4) at (1,-2.5) {};
   \node [label=right: $v'$,blackvertex,scale=0.3] (a5) at (1,-3.5) {};
   \node [label=right: $v''$,blackvertex,scale=0.3] (a6) at (1,-4.5) {};
   \node [label=right: $w$,blackvertex,scale=0.3] (a7) at (2,-2.5) {};
   \node [label=right: $w'$,blackvertex,scale=0.3] (a8) at (2,-3.5) {};
   \node [label=right: $w''$,blackvertex,scale=0.3] (a9) at (2,-4.5) {};
   \node [label=below: (c) The graph in Case 2 (3)] (a10) at (1,-5.0) {};
   \draw [dotted,thick] (a1)--(a4);
   \draw [dotted,thick] (a1)--(a5);
   \draw [dotted,thick] (a1)--(a6);
   \draw [dashdotted,thick] (a4)--(a5);
   \draw [dashdotted,thick] (a4)--(a7);
   \draw [dashdotted,thick] (a5)--(a7);
   \draw [black,thick] (a1)--(a8);
   \draw [black,thick] (a7)--(a8);
   \draw [black,thick] (a1) .. controls (1,-2.1) .. (a7);
   \draw [black,thick] (a2)--(a4);
   \draw [black,thick] (a2)--(a3);
   \draw [black,thick] (a3)--(a4);
   \draw [black,thick] (a5)--(a6);
   \draw [black,thick] (a6)--(a9);
   \draw [black,thick] (a5)--(a9);

   \node [label=left: $u$,blackvertex,scale=0.3] (a1) at (4.5,-2.5) {};
   \node [label=left: $u'$,blackvertex,scale=0.3] (a2) at (4.5,-3.5) {};
   \node [label=left: $u''$,blackvertex,scale=0.3] (a3) at (4.5,-4.5) {};
   \node [label=right: $v$,blackvertex,scale=0.3] (a4) at (5.5,-2.5) {};
   \node [label=right: $v'$,blackvertex,scale=0.3] (a5) at (5.5,-3.5) {};
   \node [label=right: $v''$,blackvertex,scale=0.3] (a6) at (5.5,-4.5) {};
   \node [label=right: $w$,blackvertex,scale=0.3] (a7) at (6.5,-2.5) {};
   \node [label=right: $w'$,blackvertex,scale=0.3] (a8) at (6.5,-3.5) {};
   \node [label=right: $w''$,blackvertex,scale=0.3] (a9) at (6.5,-4.5) {};
   \node [label=below: (d) The graph in Case 2 (4)] (a10) at (5.5,-5.0) {};
   \draw [dotted,thick] (a1)--(a4);
   \draw [dotted,thick] (a1)--(a5);
   \draw [dotted,thick] (a1) .. controls (5.5,-2.1) .. (a7);
   \draw [dashdotted,thick] (a4)--(a5);
   \draw [dashdotted,thick] (a4)--(a7);
   \draw [dashdotted,thick] (a5)--(a7);
   \draw [black,thick] (a1)--(a6);
   \draw [black,thick] (a1)--(a8);
   \draw [black,thick] (a1) .. controls (5.5,-3.9) .. (a9);
   \draw [black,thick] (a1)--(a2);
   \draw [black,thick] (a1) .. controls (3.7,-3.5) .. (a3);
   \draw [black,thick] (a2)--(a3);
   \draw [black,thick] (a2)--(a4);
   \draw [black,thick] (a2)--(a5);
   \draw [black,thick] (a2)--(a6);
   \draw [black,thick] (a2)--(a7);
   \draw [black,thick] (a2) .. controls (5.5,-3.9) .. (a8);
   \draw [black,thick] (a2)--(a9);
   \draw [black,thick] (a3)--(a4);
   \draw [black,thick] (a3)--(a5);
   \draw [black,thick] (a3)--(a6);
   \draw [black,thick] (a3) .. controls (5.5,-3.9) .. (a7);
   \draw [black,thick] (a3)--(a8);
   \draw [black,thick] (a3) .. controls (5.5,-4.9) .. (a9);
   \draw [black,thick] (a4) .. controls (5.1,-3.5) .. (a6);
   \draw [black,thick] (a4)--(a8);
   \draw [black,thick] (a4)--(a9);
   \draw [black,thick] (a5)--(a6);
   \draw [black,thick] (a5)--(a8);
   \draw [black,thick] (a5)--(a9);
   \draw [black,thick] (a6)--(a7);
   \draw [black,thick] (a6)--(a8);
   \draw [black,thick] (a6)--(a9);
   \draw [black,thick] (a7)--(a8);
   \draw [black,thick] (a8)--(a9);
   \draw [black,thick] (a7) .. controls (7.3,-3.5) .. (a9);
\end{tikzpicture}

\small Figure 4. The graphs in Case 2.
\end{center}
\end{figure}

The proof is complete. {\hfill$\Box$}

\noindent\textbf{Proof of Theorem \ref{13}.}
We will prove that every edge-coloring of $K_{n}$ with at least $\left\lfloor\frac{(n-k+2)^{2}}{4}\right\rfloor+(k-2)(n-k+2)+\binom{k-2}{2}+(k-1)^{2}-\frac{n-3k}{2}+2$ distinct colors (call this edge-colored graph $G$) contains a rainbow $kC_{3}$ by induction on $n$. When $n=3k$, the graph $G$ is edge-colored with at least $\binom{3k-1}{2}+2$ colors and we can find a rainbow $kC_{3}$ in $G$ by Theorem \ref{5}. So we assume that $n\geq3k+1$.

If there exists a vertex $u\in V(G)$ such that $d^{s}(u)\leq\left\lfloor\frac{n+k-3}{2}\right\rfloor$, then $c(G-u)=c(G)-d^{s}(u)\geq\left\lfloor\frac{(n-k+1)^{2}}{4}\right\rfloor+(k-2)(n-k+1)+\binom{k-2}{2}+(k-1)^{2}-\frac{n-3k-1}{2}+2$. By the induction hypothesis, we can find a rainbow $kC_{3}$ in $G-u$. So we assume that $d^{s}(v)\geq\left\lfloor\frac{n+k-1}{2}\right\rfloor$ for every $v\in V(G)$.

Let $H$ be a spanning subgraph of $G$ such that $E(v, H-v)$ contains exactly one edge of each color saturated by $v$ for every $v\in V(H)$. Then $H$ is rainbow and $d_{H}(v)\geq\left\lfloor\frac{n+k-1}{2}\right\rfloor\geq\frac{n+k-2}{2}$ for each $v\in V(H)$. By Lemmas \ref{11} and \ref{111}, one of the following two alternatives holds:
$(1)$~there exists a partition $V(H)=V_{1}\cup V_{2}\cup V_{3}$ such that $|V_{1}|=k-2$, $|V_{2}|=|V_{3}|=\frac{n-k+2}{2}$, any two vertices in different parts are adjacent, and there are no edges joining pairs of vertices in $V_{2}$ and so does $V_{3}$;
$(2)$~there exists a subgraph $H'\subseteq H$ such that $|V(H')|=3k-3$, $H'$ contains a $(k-1)C_{3}$ and $H-H'$ is a complete bipartite graph whose two parts are at least 2 in size.

If $(1)$ is the case then $$e(H)\leq\left\lfloor\frac{(n-k+2)^{2}}{4}\right\rfloor+(k-2)(n-k+2)+\binom{k-2}{2}.$$ This implies that $$c(G-E(H))=c(G)-e(H)\geq(k-1)^{2}-\frac{n-3k}{2}+2\geq(k-1)^{2}-\frac{2k^{2}-k+2-3k}{2}+2=2.$$ Hence there must be two new colors not in $C(H)$ appearing on $E(G[V_{2}])\cup E(G[V_{3}])$. Recall that $n\geq3k+1$. We have $$|V_{2}|-(k-2)=\frac{n-k+2}{2}-(k-2)\geq4$$
and
$$|V_{3}|-(k-2)=\frac{n-k+2}{2}-(k-2)\geq4.$$
So, we can first find a rainbow $(k-2)C_{3}$ in $H$ (say $T_{1}$), and then we can find a rainbow $2C_{3}$ in $G-V(T_{1})$ (say $T_{2}$) such that $T_{1}$ and $T_{2}$ constitute a rainbow $kC_{3}$ in $G$.

If $(2)$ is the case then it may be supposed without loss of generality that $(A, B)$ is a bipartition of $H-H'$. Hence $|A|\geq2$ and $|B|\geq2$. By the definition of $H$, either there exists an edge $e'\in E(G[A])\cup E(G[B])$ such that $C(e')\notin C(H)$ or for each edge $e\in E(G[A])\cup E(G[B])$, $C(e)$ is the color saturated by the ends of $e$. In both cases, we can find a rainbow $kC_{3}$ in $G$.

The proof is complete. {\hfill$\Box$}

\noindent\textbf{Proof of Theorem \ref{12}.}
The lower bound is due to Theorem \ref{lb}. For the upper bound, we will show that every edge-coloring of $K_{n}$ with at least  $\left\lfloor\frac{(n-k+2)^{2}}{4}\right\rfloor+(k-2)(n-k+2)+\binom{k-2}{2}+2$ distinct colors (call this edge-colored graph $G$) contains a rainbow $kC_{3}$ by induction on $n$.

When $n=2k^{2}-k+2$, we have $(k-1)^{2}-\frac{n-3k}{2}=0$. By Theorem \ref{13}, we know that $G$ contains a rainbow $kC_{3}$. If there exists a vertex $u\in V(G)$ such that $d^{s}(u)\leq\left\lfloor\frac{n+k-2}{2}\right\rfloor$, then $$c(G-u)=c(G)-d^{s}(u)\geq\left\lfloor\frac{(n-k+1)^{2}}{4}\right\rfloor+(k-2)(n-k+1)+\binom{k-2}{2}+2.$$ By the induction hypothesis, we know that $G-u$ contains a rainbow $kC_{3}$. So we assume that $d^{s}(v)\geq\left\lfloor\frac{n+k}{2}\right\rfloor\geq\frac{n+k-1}{2}$ for every $v\in V(G)$. Let $H$ be a spanning subgraph of $G$ such that $E(v, H-v)$ contains exactly one edge of each color saturated by $v$ for every $v\in V(H)$. Then $H$ is rainbow and $d_{H}(v)\geq\frac{n+k-1}{2}$ for each $v\in V(H)$. By Theorem \ref{D}, we see that $H$ contains a $(k-1)C_{3}$. Hence, there exists a subgraph $H'\subseteq H$ such that $|V(H')|=3k-3$, $H'$ contains a $(k-1)C_{3}$ and $H-H'$ is a complete bipartite graph whose two parts are at least 2 in size by Lemma \ref{111}. Let $(A, B)$ be a bipartition of $H-H'$. Then $|A|\geq2$ and $|B|\geq2$. By the definition of $H$, either there exists an edge $e'\in E(G[A])\cup E(G[B])$ such that $C(e')\notin C(H)$ or for each edge $e\in E(G[A])\cup E(G[B])$, $C(e)$ is the color saturated by the ends of $e$. In both cases, we can find a rainbow $kC_{3}$ in $G$.

The proof is complete. {\hfill$\Box$}

\section{Proofs of Lemmas \ref{11} and \ref{111}}\label{se4}
\textbf{Proof of Lemma \ref{11}.}
If $n=3k$, then $\delta(G)\geq\frac{n+k-1}{2}=\frac{4k-1}{2}$, i.e., $\delta(G)\geq2k$. By Theorem \ref{D}, we can find a $kC_{3}$ in $G$. Thus we may assume that $n\geq3k+1$.

Let $\mathcal{T}$ denote the set of all those subgraphs of $G$ which have $3k-3$ vertices and contain $k-1$ vertex-disjoint triangles and let $\mathcal{T}^{*}$ denote the set of those elements $T$ of $\mathcal{T}$ for which $G-T$ contains a path of maximal length. By Theorem \ref{D}, we know that $\mathcal{T}, \mathcal{T}^{*}\neq\emptyset$. Now, we proceed by proving the following.

\setcounter{claim}{0}
\begin{claim}\label{C11}
If $T\in\mathcal{T}$ and $uv$ is an edge of $G-T$, then $e(u, T)+e(v, T)\geq4k-4$.
\end{claim}

\begin{proof}
Since $G$ contains no $kC_{3}$'s, we see that $G-T$ contains no triangles, which implies that $d_{G-T}(u)+d_{G-T}(v)\leq n-3(k-1)=n-3k+3$. Hence $e(u, T)+e(v, T)\geq2\big(\frac{n+k-1}{2}\big)-(n-3k+3)=4k-4$.
\end{proof}

\begin{claim}\label{C12}
If $T\in\mathcal{T}^{*}$, then $G-T$ contains at least one edge.
\end{claim}

\begin{proof}
Let $T_{1}, T_{2}, \ldots, T_{k-1}$ denote $k-1$ vertex-disjoint triangles contained in $T$. Recall that $n\geq3k+1$. We have $v(G-T)=n-3(k-1)\geq4$. If $u$ and $v$ are two isolated vertices in $G-T$, then $e(u, T)+e(v, T)\geq 2\big(\frac{n+k-1}{2}\big)=n+k-1\geq4k>4(k-1)$. So there exists a triangle in $T$, say $T_{1}$, such that $e(u, T_{1})+e(v, T_{1})\geq5$. Let $V(T_{1})=\{x, y, z\}$. It may be supposed without loss of generality that $u$ is joined to $x, y, z$ and $v$ to $x, y$. Let $H$ denote $G[\{x, y, v\}\cup V(T-T_{1})]$. Then $H\in \mathcal{T}$ and $uz\in G-H$. By the definition of $\mathcal{T}^{*}$, we know that $G-T$ contains at least one edge.
\end{proof}

\begin{claim}\label{C13}
There must exist a graph $T\in\mathcal{T}$ such that $G-T$ contains a $2K_{2}$.
\end{claim}

\begin{proof}
Let $T'$ be a graph in $\mathcal{T}^{*}$ and let $T_{1}, T_{2}, \ldots, T_{k-1}$ denote $k-1$ vertex-disjoint triangles contained in $T'$. By Claim \ref{C12}, we know that $G-T'$ contains at least one edge. Suppose that $ab$ is an edge in $G-T'$.

Recall that $n\geq3k+1$. We have $v(G-T')=n-3(k-1)\geq4$. If $G-T'$ contains more than one isolated vertex, let $u$ and $v$ be two isolated vertices in $G-T'$, then $e(u, T')+e(v, T')\geq 2(\frac{n+k-1}{2})=n+k-1\geq4k>4(k-1)$. So there exists a triangle in $T'$, say $T_{1}$, such that $e(u, T_{1})+e(v, T_{1})\geq5$. Let $V(T_{1})=\{x, y, z\}$. Without loss of generality, assume that $u$ is joined to $x, y, z$ and $v$ to $x, y$. Let $H$ denote $G[\{x, y, v\}\cup V(T'-T_{1})]$. Then $H\in \mathcal{T}$ and $\{ab, uz\}$ is a $2K_{2}$ in $G-H$. So next we will assume that $G-T'$ contains at most one isolated vertex. If $G-T'$ contains no $2K_{2}$'s, then one of the following two alternatives holds:
$(1)$ $V(G-T')=\{v_{0}, v_{1}, \ldots, v_{t-1}\}$ and $E(G-T')=\{v_{0}v_{1}, v_{0}v_{2}, \ldots, v_{0}v_{t-2}\}$;
$(2)$ $V(G-T')=\{v_{0}, v_{1}, \ldots, v_{t-1}\}$ and $E(G-T')=\{v_{0}v_{1}, v_{0}v_{2}, \ldots, v_{0}v_{t-1}\}$, where $t\geq4$ (since $v(G-T')\geq4$).

If $(1)$ is the case then $e(v_{1}, T')+e(v_{t-1}, T')\geq\frac{n+k-1}{2}-1+\frac{n+k-1}{2}=n+k-2\geq4k-1>4(k-1)$, so it may be assumed without loss of generality that $e(v_{1}, T_{1})+e(v_{t-1}, T_{1})\geq5$. Then either $e(v_{1}, T_{1})=2$ or $e(v_{1}, T_{1})=3$. If $e(v_{1}, T_{1})=2$, then it may be supposed without loss of generality that $v_{1}x, v_{1}y\in E(G)$, and $e(v_{t-1}, T_{1})=3$. Let $H$ denote $G[\{x, y, v_{1}\}\cup V(T'-T_{1})]$. Then $H\in\mathcal{T}$ and $\{zv_{t-1}, v_{0}v_{2}\}$ is a $2K_{2}$ in $G-H$. If $e(v_{1}, T_{1})=3$, then $e(v_{t-1}, T_{1})\geq2$ and it may be supposed without loss of generality that $v_{t-1}x, v_{t-1}y\in E(G)$. Let $H$ denote $G[\{x, y, v_{t-1}\}\cup V(T'-T_{1})]$. Then $H\in\mathcal{T}$ and $\{zv_{1}, v_{0}v_{2}\}$ is a $2K_{2}$ in $G-H$.

If $(2)$ is the case then $e(\{v_{1}, v_{2}, v_{3}\}, T')\geq3(\frac{n+k-1}{2}-1)\geq6k-3>6(k-1)$, so it may be assumed without loss of generality that $e(\{v_{1}, v_{2}, v_{3}\}, T_{1})\geq7$, and further that $v_{1}$ is joined to $x, y, z$ and $v_{2}$ to $x, y$. Let $H$ denote $G[\{x, y, v_{2}\}\cup V(T'-T_{1})]$. Then $H\in\mathcal{T}$ and $\{zv_{1}, v_{0}v_{3}\}$ is a $2K_{2}$ in $G-H$.
\end{proof}

\begin{claim}\label{C14}
If $T\in\mathcal{T}$ such that $G-T$ contains a $2K_{2}$, let $T_{1}, T_{2}, \ldots, T_{k-1}$ denote $k-1$ vertex-disjoint triangles contained in $T$ and let $M$ denote a $2K_{2}$ in $G-T$, then $e(M, T_{i})=8$ for every $1\leq i\leq k-1$. Suppose, furthermore, that $M=\{ab, cd\}$ and $T_{i}=t_{i_{1}}t_{i_{2}}t_{i_{3}}t_{i_{1}}$. Then $G[V(M)\cup V(T_{i})]$ contains a subgraph isomorphic to one of the following graphs

(${\romannumeral1}$)~$G_{1}$: $V(G_{1})=V(M)\cup V(T_{i})$ and $E(G_{1})=E(M)\cup E(T_{i})\cup\{t_{i_{1}}a, t_{i_{1}}b, t_{i_{1}}c, t_{i_{1}}d, t_{i_{2}}a, \\t_{i_{2}}c, t_{i_{3}}b, t_{i_{3}}d\}$;

(${\romannumeral2}$)~$G_{2}$: $V(G_{2})=V(M)\cup V(T_{i})$ and $E(G_{2})=E(M)\cup E(T_{i})\cup\{t_{i_{1}}a, t_{i_{1}}b, t_{i_{1}}c, t_{i_{2}}a, t_{i_{2}}b, \\t_{i_{2}}c, t_{i_{3}}a, t_{i_{3}}d\}$;

(${\romannumeral3}$)~$G_{3}$: $V(G_{3})=V(M)\cup V(T_{i})$ and $E(G_{3})=E(M)\cup E(T_{i})\cup\{t_{i_{1}}a, t_{i_{1}}b, t_{i_{1}}c, t_{i_{2}}a, t_{i_{2}}b, \\t_{i_{2}}d, t_{i_{3}}a, t_{i_{3}}b\}$.
\end{claim}

\begin{proof}
By Claim \ref{C11}, we have $e(M, T)\geq2(4k-4)=8k-8$. Furthermore, we see that $e(M, T)\leq8(k-1)=8k-8$; otherwise there exists a triangle in $T$, say $T_{1}$, such that $e(M, T_{1})\geq9$. By Lemma \ref{l4}, we know that $G[V(M)\cup V(T_{1})]$ contains two vertex-disjoint triangles, with $T_{2}, \ldots, T_{k-1}$ as $k$ vertex-disjoint triangles contained in $G$, a contradiction. Hence we have $e(M, T)=8k-8$. Again, by Lemma \ref{l4}, one can show that $e(M, T_{i})=8$ for every $1\leq i\leq k-1$.

Since $G$ contains no $kC_{3}$'s, for each $T_{i}\subseteq T$, we have $G[V(M)\cup V(T_{i})]$ contains no $2C_{3}$'s. Hence for each pair of vertices $t_{i_{x}}$ and $t_{i_{y}}$ of $T_{i}$ with $1\leq x<y\leq3$, $e(\{t_{i_{x}}, t_{i_{y}}\}, M)\leq6$; otherwise it may be assumed without loss of generality that $e(\{t_{i_{1}}, t_{i_{2}}\}, M)\geq7$, furthermore $t_{i_{1}}$ is joined to $a, b, c, d$ and $t_{i_{2}}$ to $a, b, c$. Then $\{abt_{i_{2}}a, cdt_{i_{1}}c\}$ is a $2C_{3}$ in $G[V(M)\cup V(T_{i})]$. Recall that $e(M, T_{i})=8$. We see that there exist two vertices in $T_{i}$, say $t_{i_{1}}$ and $t_{i_{2}}$, such that $e(\{t_{i_{1}}, t_{i_{2}}\}, M)\geq6$. Hence $e(\{t_{i_{1}}, t_{i_{2}}\}, M)=6$. It may be assumed without loss of generality that one of the alternatives $e(t_{i_{1}}, M)=4$ and $e(t_{i_{1}}, M)=3$ holds.

If $e(t_{i_{1}}, M)=4$, then $e(t_{i_{2}}, M)=2$ and $e(t_{i_{3}}, M)=2$. Since $G[V(M)\cup V(T_{i})]$ contains no $2C_{3}$'s, we have $N_{M}(t_{i_{2}})\cap N_{M}(t_{i_{3}})=\emptyset$, and $|N_{M}(t_{i_{2}})\cap V(e)|=1$ and $|N_{M}(t_{i_{3}})\cap V(e)|=1$ for each $e\in E(M)$, which implies that $G[V(M)\cup V(T_{i})]$ contains a subgraph isomorphic to $G_{1}$.

If $e(t_{i_{1}}, M)=3$, then $e(t_{i_{2}}, M)=3$ and $e(t_{i_{3}}, M)=2$. Without loss of generality, assume that $N_{M}(t_{i_{1}})=\{a, b, c\}$. This implies that $a, b\in N_{M}(t_{i_{2}})$; otherwise $\{abt_{i_{1}}a, cdt_{i_{2}}c\}$ is a $2C_{3}$ in $G[V(M)\cup V(T_{i})]$. If $N_{M}(t_{i_{2}})=\{a, b, c\}$, then $N_{M}(t_{i_{3}})\in\{\{a, d\}, \{b, d\}\}$, which implies that $G[V(M)\cup V(T_{i})]$ contains a subgraph isomorphic to $G_{2}$. If $N_{M}(t_{i_{2}})=\{a, b, d\}$, then $N_{M}(t_{i_{3}})=\{a, b\}$, which implies that $G[V(M)\cup V(T_{i})]$ contains a subgraph isomorphic to $G_{3}$.
\end{proof}

\begin{claim}\label{C15}
If $T\in\mathcal{T}$ such that $G-T$ contains a $2K_{2}$, then $G-T$ is a bipartite graph.
\end{claim}

\begin{proof}
By contradiction. Let $C=v_{1}v_{2}\cdots v_{l}v_{1}$ (indices are taken modulo $l$) be an odd cycle in $G-T$. Since $G$ contains no $kC_{3}$'s, we know that $G-T$ contains no $C_{3}$'s which implies that $l\geq5$. Let $T_{1}, T_{2}, \ldots, T_{k-1}$ denote $k-1$ vertex-disjoint triangles contained in $T$. By Claim \ref{C14}, we have $e(\{v_{i}v_{i+1}, v_{i+2}v_{i+3}\}, T_{j})=8$ for every $1\leq i\leq l$ and $1\leq j\leq k-1$.

If there exists $T_{j'}\subseteq T$ such that $G[\{v_{i'}, v_{i'+1}, v_{i'+2}, v_{i'+3}\}\cup V(T_{j'})]$ contains a subgraph isomorphic to $G_{2}$ for some $i'$, $1\leq i'\leq l$, then by Claim \ref{C14}, it may be assumed without loss of generality that $e(v_{i'}v_{i'+1}, T_{j'})=5$ and $e(v_{i'+2}v_{i'+3}, T_{j'})=3$. This implies that $e(v_{i'+3}v_{i'+4}, T_{j'})=3$ and $e(v_{i'+4}v_{i'+5}, T_{j'})=5$, furthermore $e(v_{i'+1}v_{i'+2}, T_{j'})=3$. This contradicts that $e(\{v_{i'+1}v_{i'+2}, v_{i'+3}v_{i'+4}\}, T_{j'})=8$.

Similarly, if there exists $T_{j'}\subseteq T$ such that $G[\{v_{i'}, v_{i'+1}, v_{i'+2}, v_{i'+3}\}\cup V(T_{j'})]$ contains a subgraph isomorphic to $G_{3}$ for some $i'$, $1\leq i'\leq l$, then by Claim \ref{C14}, it may be assumed without loss of generality that $e(v_{i'}v_{i'+1}, T_{j'})=6$ and $e(v_{i'+2}v_{i'+3}, T_{j'})=2$. This implies that $e(v_{i'+3}v_{i'+4}, T_{j'})=2$ and $e(v_{i'+4}v_{i'+5}, T_{j'})=6$, furthermore $e(v_{i'+1}v_{i'+2}, T_{j'})=2$. Hence $e(\{v_{i'+1}v_{i'+2}, v_{i'+3}v_{i'+4}\}, T_{j'})=4$, a contradiction.

From the above analysis, we see that $G[\{v_{i}, v_{i+1}, v_{i+2}, v_{i+3}\}\cup V(T_{j})]$ contains a subgraph isomorphic to $G_{1}$ for all $1\leq i\leq l$ and $1\leq j\leq k-1$. Let $T_{1}=xyzx$ and $M=\{v_{1}v_{2}, v_{3}v_{4}\}$. Without loss of generality, assume that $N_{M}(x)=\{v_{1}, v_{2}, v_{3}, v_{4}\}$, $N_{M}(y)=\{v_{1}, v_{3}\}$ and $N_{M}(z)=\{v_{2}, v_{4}\}$. Then $N_{C}(x)=V(C)$. Since $C$ is an odd cycle and $G[\{v_{i}, v_{i+1}, v_{i+2}, v_{i+3}\}\cup V(T_{j})]$ contains a subgraph isomorphic to $G_{1}$ for all $1\leq i\leq l$ and $1\leq j\leq k-1$, there exist two consecutive vertices in $C$, say $v_{s}$ and $v_{s+1}$, such that $\{v_{s}, v_{s+1}\}\subset N_{C}(y)$ or $\{v_{s}, v_{s+1}\}\subset N_{C}(z)$. Now $xv_{s-1}v_{s-2}x$, one of $\{yv_{s}v_{s+1}y, zv_{s}v_{s+1}z\}$ and $T_{2}, \ldots, T_{k-1}$ constitute $k$ vertex-disjoint triangles contained in $G$, a contradiction.
\end{proof}

\begin{claim}\label{C16-}
Let $T\in\mathcal{T}$ such that $G-T$ contains a $2K_{2}$, and let $T_{1}, T_{2}, \ldots, T_{k-1}$ denote $k-1$ vertex-disjoint triangles contained in $T$. If $e(v, T_{i})\leq2$ for each $v\in V(G-T)$ and $T_{i}\subseteq T$, then there exists a partition $V(G)=V_{1}\cup V_{2}\cup V_{3}$, so that $|V_{1}|=k-1$, $|V_{2}|=|V_{3}|=\frac{n-k+1}{2}$, any two vertices in different parts are adjacent, and there are no edges joining pairs of vertices in $V_{2}$ and so does $V_{3}$.
\end{claim}

\begin{proof}
For each $v\in V(G-T)$ and $T_{i}\subseteq T$, since $e(v, T_{i})\leq2$, we have $d_{G-T}(v)\geq\frac{n+k-1}{2}-2(k-1)=\frac{n-3k+3}{2}$. Recall that $v(G-T)=n-3(k-1)=n-3k+3$. By Claim \ref{C15}, we see that $G-T=K_{\frac{n-3k+3}{2}, \frac{n-3k+3}{2}}$, which implies that all the equalities are attained in above, that is, $e(v, T_{i})=2$ for every $v\in V(G-T)$ and $T_{i}\subseteq T$. Let $(A, B)$ be a bipartition of $G-T$. Then $|A|=\frac{n-3k+3}{2}\geq2$ and $|B|=\frac{n-3k+3}{2}\geq2$ since $n\geq3k+1$. For any subgraph $M=2K_{2}$ of $G-T$ and every $T_{i}\in T$, by Claim \ref{C14}, we know that $G[V(M)\cup V(T_{i})]$ contains a subgraph isomorphic to $G_{1}$. Let $V(T_{i})=\{t_{i_{1}}, t_{i_{2}}, t_{i_{3}}\}$. Without loss of generality, assume that $e(t_{i_{1}}, M)=4$ for every $1\leq i\leq k-1$. Then by $G-T=K_{\frac{n-3k+3}{2}, \frac{n-3k+3}{2}}$, $e(t_{i_{1}}, G-T)=v(G-T)=n-3k+3$ with $1\leq i\leq k-1$. If there exist $a\in A$ and $b\in B$ such that $a, b\in N_{G-T}(v')$ for some $v'\in \{t_{1_{2}}, t_{2_{2}}, \ldots, t_{(k-1)_{2}}, t_{1_{3}}, t_{2_{3}}, \ldots, t_{(k-1)_{3}}\}$, it may be assumed without loss of generality that $v'=t_{1_{2}}$, then $abt_{1_{2}}a$, $t_{1_{1}}cdt_{1_{1}}$ and $T_{2}, \ldots, T_{k-1}$ constitute $k$ vertex-disjoint triangles contained in $G$ for each $cd\in E(G-T-\{a, b\})$, a contradiction. Hence for all $v\in\{t_{1_{2}}, t_{2_{2}}, \ldots, t_{(k-1)_{2}}, t_{1_{3}}, t_{2_{3}}, \ldots, t_{(k-1)_{3}}\}$, either $N_{G-T}(v)\subseteq A$ or $N_{G-T}(v)\subseteq B$. Since $G-T=K_{\frac{n-3k+3}{2}, \frac{n-3k+3}{2}}$ and $G[V(M)\cup V(T_{i})]$ contains a subgraph isomorphic to $G_{1}$ for any subgraph $M=2K_{2}$ of $G-T$ and every $T_{i}\in T$, either $N_{G-T}(t_{i_{2}})=A$ and $N_{G-T}(t_{i_{3}})=B$ or $N_{G-T}(t_{i_{2}})=B$ and $N_{G-T}(t_{i_{3}})=A$, where $1\leq i\leq k-1$.

Next, we will show that $G-\{t_{1_{1}}, t_{2_{1}}, \ldots, t_{(k-1)_{1}}\}$ is a bipartite graph. Suppose the contrary. Let $C=v_{1}v_{2}\cdots v_{l}v_{1}$ be the shortest odd cycle in $G-\{t_{1_{1}}, t_{2_{1}}, \ldots, t_{(k-1)_{1}}\}$. Since $G-T=K_{\frac{n-3k+3}{2}, \frac{n-3k+3}{2}}$, we have $V(C)\cap\{t_{1_{2}}, t_{2_{2}}, \ldots, t_{(k-1)_{2}}, t_{1_{3}}, t_{2_{3}}, \ldots, t_{(k-1)_{3}}\}\neq\emptyset$. We distinguish three cases.

\setcounter{case}{0}
\begin{case}
$V(C)\subset\{t_{1_{2}}, t_{2_{2}}, \ldots, t_{(k-1)_{2}}, t_{1_{3}}, t_{2_{3}}, \ldots, t_{(k-1)_{3}}\}$.
\end{case}

In this case, there exist two consecutive vertices in $C$, say $v_{x}$ and $v_{x+1}$, such that either $N_{G-T}(v_{x})=N_{G-T}(v_{x+1})=A$ or $N_{G-T}(v_{x})=N_{G-T}(v_{x+1})=B$. If $l\geq5$, then $av_{x}v_{x+1}a$ or $bv_{x}v_{x+1}b$ is a shorter odd cycle, where $a\in A$ and $b\in B$, a contradiction. Hence we have $l=3$. Since for every $1\leq i\leq k-1$, either $N_{G-T}(t_{i_{2}})=A$ and $N_{G-T}(t_{i_{3}})=B$ or $N_{G-T}(t_{i_{2}})=B$ and $N_{G-T}(t_{i_{3}})=A$, it may be assumed without loss of generality that one of the following two alternatives holds: $(1)$ $V(C)=\{t_{1_{2}}, t_{2_{2}}, t_{2_{3}}\}$, and $N_{G-T}(t_{1_{2}})=N_{G-T}(t_{2_{2}})=A$ and $N_{G-T}(t_{2_{3}})=B$;
$(2)$ $V(C)=\{t_{1_{2}}, t_{2_{2}}, t_{3_{2}}\}$ and $N_{G-T}(t_{1_{2}})=N_{G-T}(t_{2_{2}})=N_{G-T}(t_{3_{2}})=A$. Suppose that $a_{1}, a_{2}\in A$ and $b_{1}, b_{2}\in B$. If $(1)$ is the case then $a_{1}t_{1_{2}}t_{2_{2}}a_{1}$, $b_{1}t_{2_{1}}t_{2_{3}}b_{1}$, $b_{2}t_{1_{1}}t_{1_{3}}b_{2}$ and $T_{3}, \ldots, T_{k-1}$ constitute $k$ vertex-disjoint triangles contained in $G$, a contradiction. If $(2)$ is the case then $a_{1}t_{1_{2}}t_{2_{2}}a_{1}$, $b_{1}t_{2_{1}}t_{2_{3}}b_{1}$, $a_{2}t_{3_{1}}t_{3_{2}}a_{2}$, $b_{2}t_{1_{1}}t_{1_{3}}b_{2}$ and $T_{4}, \ldots, T_{k-1}$ constitute $k$ vertex-disjoint triangles contained in $G$, a contradiction.

\begin{case}
There exists an edge $ab\in E(G-T)$ such that $a, b\in V(C)$.
\end{case}

Since $C$ is the shortest odd cycle in $G-\{t_{1_{1}}, t_{2_{1}}, \ldots, t_{(k-1)_{1}}\}$, we see that $ab\in E(C)$. Without loss of generality, assume that $ab=v_{x}v_{x+1}$. Recall that $V(C)\cap\{t_{1_{2}}, t_{2_{2}}, \ldots, t_{(k-1)_{2}}, \\t_{1_{3}}, t_{2_{3}}, \ldots, t_{(k-1)_{3}}\}\neq\emptyset$ and for each $v\in\{t_{1_{2}}, t_{2_{2}}, \ldots, t_{(k-1)_{2}}, t_{1_{3}}, t_{2_{3}}, \ldots, t_{(k-1)_{3}}\}$, either $N_{G-T}(v)=A$ or $N_{G-T}(v)=B$. We have $l\geq5$ and $V(C)\cap\{t_{1_{2}}, t_{2_{2}}, \ldots, t_{(k-1)_{2}}, t_{1_{3}}, t_{2_{3}}, \ldots, \\t_{(k-1)_{3}}\}\subseteq\{v_{x-1}, v_{x+2}\}$; otherwise we can find a shorter odd cycle in $G-\{t_{1_{1}}, t_{2_{1}}, \ldots, t_{(k-1)_{1}}\}$. This implies that $|V(C)\setminus\{t_{1_{2}}, t_{2_{2}}, \ldots, t_{(k-1)_{2}}, t_{1_{3}}, t_{2_{3}}, \ldots, t_{(k-1)_{3}}\}|\geq3$. Hence there exists a vertex $u\in V(G-T-\{a, b\})$ such that $u\in V(C)$. Since $G-T=K_{\frac{n-3k+3}{2}, \frac{n-3k+3}{2}}$, $v_{x}v_{x+1}\ldots uv_{x}$ or $v_{x+1}v_{x+2}\ldots uv_{x+1}$ is a shorter odd cycle in $G-\{t_{1_{1}}, t_{2_{1}}, \ldots, t_{(k-1)_{1}}\}$, a contradiction.

\begin{case}
$V(C)\cap V(G-T)\neq\emptyset$, and $V(C)\cap V(G-T)\subseteq A$ or $V(C)\cap V(G-T)\subseteq B$.
\end{case}

Without loss of generality, assume that $V(C)\cap V(G-T)\subseteq A$. Since $V(C)\cap V(G-T)\neq\emptyset$, we have $|V(C)\cap V(G-T)|\geq1$. If $|V(C)\cap V(G-T)|\geq2$, then there exist vertices $a_{1}, a_{2}\in A$ such that $a_{1}, a_{2}\in V(C)$ and $a_{1}a_{2}\notin E(C)$, which implies that $l\geq5$. Furthermore, we have $N_{C}(a_{1}), N_{C}(a_{2})\subset\{t_{1_{2}}, t_{2_{2}}, \ldots, t_{(k-1)_{2}}, t_{1_{3}}, t_{2_{3}}, \ldots, t_{(k-1)_{3}}\}$. Since for each vertex $v\in\{t_{1_{2}}, t_{2_{2}}, \ldots, t_{(k-1)_{2}}, t_{1_{3}}, t_{2_{3}}, \ldots, t_{(k-1)_{3}}\}$, either $N_{G-T}(v)=A$ or $N_{G-T}(v)=B$, we can find a shorter odd cycle in $G-\{t_{1_{1}}, t_{2_{1}}, \ldots, t_{(k-1)_{1}}\}$, a contradiction. So we have $|V(C)\cap V(G-T)|=1$. Without loss of generality, assume that $V(C)\cap V(G-T)=\{v_{1}\}$. Then $V(C)\setminus\{v_{1}\}\subset\{t_{1_{2}}, t_{2_{2}}, \ldots, t_{(k-1)_{2}}, t_{1_{3}}, t_{2_{3}}, \ldots, t_{(k-1)_{3}}\}$.

If $l=3$, then $N_{G-T}(v_{2})=N_{G-T}(v_{3})=A$, since for every $1\leq i\leq k-1$, either $N_{G-T}(t_{i_{2}})=A$ and $N_{G-T}(t_{i_{3}})=B$ or $N_{G-T}(t_{i_{2}})=B$ and $N_{G-T}(t_{i_{3}})=A$. Furthermore, it may be assumed without loss of generality that $v_{2}=t_{1_{2}}$ and $v_{3}=t_{2_{2}}$. Suppose that $a_{1}, a_{2}\in A$, $b_{1}, b_{2}\in B$ and $a_{1}=v_{1}$. Now $a_{1}t_{1_{2}}t_{2_{2}}a_{1}$, $b_{1}t_{1_{1}}t_{1_{3}}b_{1}$, $b_{2}t_{2_{1}}t_{2_{3}}b_{2}$ and $T_{3}, \ldots, T_{k-1}$ constitute $k$ vertex-disjoint triangles contained in $G$, a contradiction.

If $l\geq5$, then there exist two consecutive vertices in $V(C)\setminus\{v_{1}\}$, say $v_{x}$ and $v_{x+1}$, such that either $N_{G-T}(v_{x})=N_{G-T}(v_{x+1})=A$ or $N_{G-T}(v_{x})=N_{G-T}(v_{x+1})=B$. Now $av_{x}v_{x+1}a$ or $bv_{x}v_{x+1}b$ is a shorter odd cycle, where $a\in A$ and $b\in B$, a contradiction.
\end{proof}

\begin{claim}\label{C16}
Let $T\in\mathcal{T}$ such that $G-T$ contains a $2K_{2}$, and let $T_{1}, T_{2}, \ldots, T_{k-1}$ denote $k-1$ vertex-disjoint triangles contained in $T$. If $n\geq3k+2$, then $e(v, T_{i})\leq2$ for each $v\in V(G-T)$ and $T_{i}\subseteq T$.
\end{claim}

\begin{proof}
If $n=3k+2$, then $\delta(G)\geq\frac{n+k-1}{2}=\frac{4k+1}{2}$, i.e., $\delta(G)\geq2k+1$. By Theorem \ref{D}, we can find a $kC_{3}$ in $G$. Thus we may assume that $n\geq3k+3$.

Suppose to the contrary that there exists a vertex $u\in V(G-T)$ such that $e(u, T_{i'})=3$ for some $T_{i'}\in T$. Recall that $n\geq3k+3$. We have $v(G-T-u)=n-3(k-1)-1\geq5$. Since $G-T$ contains a $2K_{2}$, we see that $G-T-u$ contains at least one edge. Let $V(T_{i})=\{t_{i_{1}}, t_{i_{2}}, t_{i_{3}}\}$, where $1\leq i\leq k-1$. Then for any edge $ab\in E(G-T-u)$, we have $e(ab, T_{i'})\leq3$; otherwise there exists a vertex in $V(T_{i'})$, say $t_{i'_{1}}$, such that $abt_{i'_{1}}a$ is a triangle, together with $ut_{i'_{2}}t_{i'_{3}}u$, and $\{T_{1}, T_{2}, \ldots, T_{k-1}\}\setminus\{T_{i'}\}$ constituting $k$ vertex-disjoint triangles contained in $G$, a contradiction. If $G-T-u$ contains a $2K_{2}$, say $M$, then $e(M, T_{i'})\leq6$, contrary to Claim \ref{C14}. So $G-T-u$ contains no $2K_{2}$'s. Since $G-T$ contains a $2K_{2}$, we see that $u$ is contained in every $2K_{2}$ of $G-T$.

We assert that $G-T-u$ does not contain more than one isolated vertex. Otherwise, suppose that $x$ and $y$ are two isolated vertices in $G-T-u$. Then $e(x, T)\geq\frac{n+k-1}{2}-1\geq2k>2(k-1)$ and $e(y, T)\geq\frac{n+k-1}{2}-1\geq2k>2(k-1)$. It may be supposed without loss of generality that $e(x, T_{j'})=3$ and $e(y, T_{j''})=3$, where $T_{j'}, T_{j''}\in T$. Similar to the discussion of $u$, we have for any edge $e$ of $G-T-x$, $e(e, T_{j'})\leq3$ and for any edge $f$ of $G-T-y$, $e(f, T_{j''})\leq3$. Since $u$ is contained in every $2K_{2}$ of $G-T$, we see that $G-T-x$ contains a $2K_{2}$, say $M_{1}$, or $G-T-y$ contains a $2K_{2}$, say $M_{2}$, which implies that $e(M_{1}, T_{j'})\leq6$ or $e(M_{2}, T_{j''})\leq6$, contrary to Claim \ref{C14}.

From the above analysis, one of the following two alternatives holds:
$(1)$ $V(G-T-u)=\{v_{0}, v_{1}, \ldots, v_{t-1}\}$
and $E(G-T-u)=\{v_{0}v_{1}, v_{0}v_{2}, \ldots, v_{0}v_{t-2}\}$;
$(2)$ $V(G-T-u)=\{v_{0}, v_{1}, \ldots, v_{t-1}\}$
and $E(G-T-u)=\{v_{0}v_{1}, v_{0}v_{2}, \ldots, v_{0}v_{t-1}\}$, where $t\geq5$.
By Claim \ref{C15}, we know that $G-T$ is a bipartite graph. Hence $d_{G-T}(v_{i})\leq2$ for every $1\leq i\leq t-1$. Since $G-T$ contains a $2K_{2}$, we have $uv_{0}\notin E(G)$. This implies that there exists a vertex $v_{i'}$ with $1\leq i'\leq t-1$ such that $G-T-v_{i'}$ contains a $2K_{2}$ in Cases (1) and (2), denote by $M'$. Note that $e(v_{i'}, T)\geq\frac{n+k-1}{2}-2\geq2k-1>2(k-1)$. So it may be assumed without loss of generality that $e(v_{i'}, T_{1})=3$. Let $M'=\{ab, cd\}$. Then $e(ab, T_{1})\leq3$; otherwise there exists a vertex in $V(T_{1})$, say $t_{1_{1}}$, such that $abt_{1_{1}}a$ is a triangle, together with $v_{i'}t_{1_{2}}t_{1_{3}}v_{i'}, T_{2}, \ldots, T_{k-1}$ constituting $k$ vertex-disjoint triangles contained in $G$, a contradiction. Similarly, we also have $e(cd, T_{1})\leq3$. Hence one can show that $e(M', T_{1})\leq6$, contrary to Claim \ref{C14}.
\end{proof}

Now, let $T$ be a graph in $\mathcal{T}$ such that $G-T$ contains a $2K_{2}$ and let $T_{1}, T_{2}, \ldots, T_{k-1}$ denote $k-1$ vertex-disjoint triangles contained in $T$. If $n\geq3k+2$, then by Claims \ref{C16-} and \ref{C16}, the lemma holds. For $n=3k+1$, by Claim \ref{C15} and $G-T$ contains a $2K_{2}$, we assume that $(\{a_{1}, a_{2}\}, \{b_{1}, b_{2}\})$ is a bipartition of $G-T$. Furthermore, by Claim \ref{C14}, one can show that
$$\sum\nolimits_{v\in V(G-T)}d_{G-T}(v)\geq\frac{4(3k+1+k-1)}{2}-8(k-1)=8.$$
Hence we have $G-T=K_{2,2}$. If $e(v, T_{i})\leq2$ for each $v\in V(G-T)$ and $T_{i}\subseteq T$, then by Claim \ref{C16-}, the lemma holds. So it may be assumed without loss of generality that $e(a_{1}, T_{1})=3$. By Claim \ref{C14}, we see that $G[\{a_{1}, a_{2}, b_{1}, b_{2}\}\cup V(T_{1})]$ contains a subgraph isomorphic to $G_{2}$. Otherwise, $G[\{a_{1}, a_{2}, b_{1}, b_{2}\}\cup V(T_{1})]$ contains two vertex-disjoint triangles, together with $T_{2}, \ldots, T_{k-1}$ constituting $k$ vertex-disjoint triangles contained in $G$, a contradiction. For every $T_{i}\in T$, let $T_{i}=t_{i_{1}}t_{i_{2}}t_{i_{3}}t_{i_{1}}$. Then we can assume that $N_{T_{1}}(a_{1})=\{t_{1_{1}}, t_{1_{2}}, t_{1_{3}}\}$, $N_{T_{1}}(a_{2})=\{t_{1_{3}}\}$, $N_{T_{1}}(b_{1})=\{t_{1_{1}}, t_{1_{2}}\}$ and $N_{T_{1}}(b_{2})=\{t_{1_{1}}, t_{1_{2}}\}$. This implies that $e(a_{2}, T-T_{1})\geq\frac{3k+1+k-1}{2}-(2+1)=2k-3>2(k-2)$, so it may be assumed without loss of generality that $e(a_{2}, T_{2})=3$. Similarly, one can show that $G[\{a_{1}, a_{2}, b_{1}, b_{2}\}\cup V(T_{2})]$ contains a subgraph isomorphic to $G_{2}$. By Claim \ref{C14}, we can assume that $N_{T_{2}}(a_{1})=\{t_{2_{1}}, t_{2_{2}}, t_{2_{3}}\}$, $N_{T_{2}}(a_{2})=\{t_{2_{3}}\}$, $N_{T_{2}}(b_{1})=\{t_{2_{1}}, t_{2_{2}}\}$ and $N_{T_{2}}(b_{2})=\{t_{2_{1}}, t_{2_{2}}\}$. First, replace $T_{1}$ by $t_{1_{1}}a_{1}b_{1}t_{1_{1}}$ and replace $G-T$ by $t_{1_{2}}t_{1_{3}}a_{2}b_{2}t_{1_{2}}$. Then one can show that $t_{1_{3}}t_{2_{1}}, t_{1_{3}}t_{2_{2}}, t_{1_{2}}t_{2_{3}}\in G$. Second, we replace $T_{1}$ by $t_{1_{2}}a_{1}b_{2}t_{1_{2}}$ and replace $G-T$ by $t_{1_{1}}t_{1_{3}}a_{2}b_{1}t_{1_{1}}$. Then we have $t_{1_{1}}t_{2_{3}}\in G$.

If $k=3$, that is, $n=10$, then $G$ is isomorphic to $G'$ in Figure $1$. If $k\geq4$, that is, $n\geq13$, then there must be two triangles in $T$, say $T_{1}$ and $T_{2}$, such that $G[V(T_{1})\cup V(T_{2})\cup V(G-T)]$ is isomorphic to $G'$ in Figure $1$. From the above analysis, we know that for each $T_{i}$ with $3\leq i\leq k-1$, $G[\{a_{1}, a_{2}, b_{1}, b_{2}\}\cup V(T_{i})]$ contains either a subgraph isomorphic to $G_{1}$ or a subgraph isomorphic to $G_{2}$.

If there exists $T_{i'}\in T$ with $3\leq i'\leq k-1$ such that $G[\{a_{1}, a_{2}, b_{1}, b_{2}\}\cup V(T_{i'})]$ contains a subgraph isomorphic to $G_{1}$, then by Claim \ref{C14}, it may be supposed without loss of generality that $N_{G-T}(t_{i'_{1}})=\{a_{1}, a_{2}, b_{1}, b_{2}\}$, $N_{G-T}(t_{i'_{2}})=\{a_{1}, a_{2}\}$ and $N_{G-T}(t_{i'_{3}})=\{b_{1}, b_{2}\}$. Replace $T_{i'}$ by $t_{i'_{1}}a_{2}b_{2}t_{i'_{1}}$ and replace $G-T$ by $a_{1}b_{1}t_{i'_{3}}t_{i'_{2}}a_{1}$. Then one can show that $t_{i'_{2}}t_{1_{1}}, t_{i'_{2}}t_{1_{2}}, t_{i'_{3}}t_{1_{3}}, t_{i'_{2}}t_{2_{1}}, t_{i'_{2}}t_{2_{2}}, t_{i'_{3}}t_{2_{1}}, t_{i'_{3}}t_{2_{2}}, t_{i'_{3}}t_{2_{3}}\in G$. Now $T_{1}$, $a_{1}t_{i'_{1}}t_{i'_{2}}a_{1}$, $a_{2}b_{2}t_{2_{1}}a_{2}$, $t_{i'_{3}}t_{2_{2}}t_{2_{3}}t_{i'_{3}}$ and $\{T_{3}, \ldots, T_{k-1}\}\backslash\{T_{i'}\}$ constitute $k$ vertex-disjoint triangles contained in $G$, a contradiction.

If there exists $T_{i'}\in T$ with $3\leq i'\leq k-1$ such that $G[\{a_{1}, a_{2}, b_{1}, b_{2}\}\cup V(T_{i'})]$ contains a subgraph isomorphic to $G_{2}$, then by Claim \ref{C14}, it may be supposed without loss of generality that $N_{G-T}(t_{i'_{1}})=\{a_{1}, a_{2}, b_{1}\}$, $N_{G-T}(t_{i'_{2}})=\{a_{1}, a_{2}, b_{1}\}$ and $N_{G-T}(t_{i'_{3}})=\{b_{1}, b_{2}\}$. Replace $T_{i'}$ by $t_{i'_{1}}a_{2}b_{1}t_{i'_{1}}$ and replace $G-T$ by $a_{1}b_{2}t_{i'_{3}}t_{i'_{2}}a_{1}$. Then one can show that $t_{i'_{2}}t_{1_{1}}, t_{i'_{2}}t_{1_{2}}, t_{i'_{3}}t_{1_{3}}, t_{i'_{2}}t_{2_{1}}, t_{i'_{2}}t_{2_{2}}, t_{i'_{3}}t_{2_{1}}, t_{i'_{3}}t_{2_{2}}, t_{i'_{3}}t_{2_{3}}\in G$. Now $T_{1}$, $b_{1}t_{i'_{1}}t_{i'_{2}}b_{1}$, $a_{2}b_{2}t_{2_{1}}a_{2}$, $t_{i'_{3}}t_{2_{2}}t_{2_{3}}t_{i'_{3}}$ and $\{T_{3}, \ldots, T_{k-1}\}\backslash\{T_{i'}\}$ constitute $k$ vertex-disjoint triangles contained in $G$, a contradiction.

The proof of Lemma \ref{11} is complete. {\hfill$\Box$}

\noindent\textbf{Proof of Lemma \ref{111}.}
Let $\mathcal{T}$ denote the set of all those subgraphs of $G$ which have $3k-3$ vertices and contain $k-1$ vertex-disjoint triangles and let $\mathcal{T}^{*}$ denote the set of those elements $T$ of $\mathcal{T}$ for which $G-T$ contains a path of maximal length. Since $G$ contains a $(k-1)C_{3}$, we see that $\mathcal{T}, \mathcal{T}^{*}\neq\emptyset$. Now, we proceed by proving the following claims.

\setcounter{claim}{0}
\begin{claim}\label{C17}
If $T\in\mathcal{T}$ and $uv$ is an edge of $G-T$, then $e(u, T)+e(v, T)\geq4k-5$.
\end{claim}

\begin{proof}
Since $G$ contains no $kC_{3}$'s, we see that $G-T$ contains no triangles, which implies that $d_{G-T}(u)+d_{G-T}(v)\leq n-3(k-1)=n-3k+3$. Hence $e(u, T)+e(v, T)\geq2(\frac{n+k-2}{2})-(n-3k+3)=4k-5$.
\end{proof}

\begin{claim}\label{C18}
If $T\in\mathcal{T}^{*}$, then $G-T$ contains at least one edge.
\end{claim}

\begin{proof}
Let $T_{1}, T_{2}, \ldots, T_{k-1}$ denote $k-1$ vertex-disjoint triangles contained in $T$. Recall that $n\geq3k+1$. We have $v(G-T)=n-3(k-1)\geq4$. If $u$ and $v$ are two isolated vertices in $G-T$, then $e(u, T)+e(v, T)\geq 2(\frac{n+k-2}{2})=n+k-2\geq4k-1>4(k-1)$. So there exists a triangle in $T$, say $T_{1}$, such that $e(u, T_{1})+e(v, T_{1})\geq5$. Let $V(T_{1})=\{x, y, z\}$. It may be supposed without loss of generality that $u$ is joined to $x, y, z$ and $v$ to $x, y$. Let $H$ denote $G[\{x, y, v\}\cup V(T-T_{1})]$. Then $H\in \mathcal{T}$ and $uz\in G-H$. By the definition of $\mathcal{T}^{*}$, we know that $G-T$ contains at least one edge.
\end{proof}

\begin{claim}\label{C19}
There must exist a graph $T\in\mathcal{T}$ such that $G-T$ contains a $2K_{2}$.
\end{claim}

\begin{proof}
Let $T'$ be a graph in $\mathcal{T}^{*}$ and let $T_{1}, T_{2}, \ldots, T_{k-1}$ denote $k-1$ vertex-disjoint triangles contained in $T'$. By Claim \ref{C18}, we see that $G-T'$ contains at least one edge. Suppose that $ab$ is an edge in $G-T'$.

If $G-T'$ contains more than one isolated vertex, then let $u$ and $v$ be two isolated vertices in $G-T'$. Hence $e(u, T')+e(v, T')\geq n+k-2\geq4k-1>4(k-1)$. So it may be assumed without loss of generality that $e(u, T_{1})+e(v, T_{1})\geq5$. Let $V(T_{1})=\{x, y, z\}$. Without loss of generality, assume that $u$ is joined to $x, y, z$ and $v$ to $x, y$. Let $H$ denote $G[\{x, y, v\}\cup V(T'-T_{1})]$. Then $H\in \mathcal{T}$ and $\{ab, uz\}$ is a $2K_{2}$ in $G-H$. So next we will assume that $G-T'$ contains at most one isolated vertex. If $G-T'$ contains no $2K_{2}$'s, then one of the following two alternatives holds: $(1)$ $V(G-T')=\{v_{0}, v_{1}, \ldots, v_{t-1}\}$
and $E(G-T')=\{v_{0}v_{1}, v_{0}v_{2}, \ldots, v_{0}v_{t-2}\}$; $(2)$ $V(G-T')=\{v_{0}, v_{1}, \ldots, v_{t-1}\}$
and $E(G-T')=\{v_{0}v_{1}, v_{0}v_{2}, \ldots, v_{0}v_{t-1}\}$, where $t\geq4$ (since $v(G-T')=n-3(k-1)\geq4$).

If $(1)$ is the case then $e(v_{1}, T')+e(v_{t-1}, T')\geq\frac{n+k-2}{2}-1+\frac{n+k-2}{2}=n+k-3\geq4k-2>4(k-1)$, so it may be assumed without loss of generality that $e(v_{1}, T_{1})+e(v_{t-1}, T_{1})\geq5$. Either $e(v_{1}, T_{1})=2$ or $e(v_{1}, T_{1})=3$. If $e(v_{1}, T_{1})=2$ then it may be supposed without loss of generality that $v_{1}x, v_{1}y\in E(G)$, and $e(v_{t-1}, T_{1})=3$. Let $H$ denote $G[\{x, y, v_{1}\}\cup V(T'-T_{1})]$. Then $H\in\mathcal{T}$ and $\{zv_{t-1}, v_{0}v_{2}\}$ is a $2K_{2}$ in $G-H$. If $e(v_{1}, T_{1})=3$, then $e(v_{t-1}, T_{1})\geq2$ and it may be supposed without loss of generality that $v_{t-1}x, v_{t-1}y\in E(G)$. Let $H$ denote $G[\{x, y, v_{t-1}\}\cup V(T'-T_{1})]$. Then $H\in\mathcal{T}$ and $\{zv_{1}, v_{0}v_{2}\}$ is a $2K_{2}$ in $G-H$.

If $(2)$ is the case then $e(\{v_{1}, v_{2}, v_{3}\}, T')\geq3(\frac{n+k-2}{2}-1)\geq6k-5>6(k-1)$, so it may be assumed without loss of generality that $e(\{v_{1}, v_{2}, v_{3}\}, T_{1})\geq7$, and further that $v_{1}$ is joined to $x, y, z$ and $v_{2}$ to $x, y$. Let $H$ denote $G[\{x, y, v_{2}\}\cup V(T'-T_{1})]$. Then $H\in\mathcal{T}$ and $\{zv_{1}, v_{0}v_{3}\}$ is a $2K_{2}$ in $G-H$.
\end{proof}

\begin{claim}\label{C110}
If $T\in\mathcal{T}$ such that $G-T$ contains a $2K_{2}$, then $G-T$ is a bipartite graph.
\end{claim}

\begin{proof}
By contradiction. Let $C=v_{1}v_{2}\cdots v_{l}v_{1}$ (indices are taken modulo $l$) be the shortest odd cycle in $G-T$. Since $G$ contains no $kC_{3}$'s, we know that $G-T$ contains no triangles which implies that $l\geq5$. Let $T_{1}, T_{2}, \ldots, T_{k-1}$ denote $k-1$ vertex-disjoint triangles contained in $T$. Then we have $e(\{v_{i}v_{i+1}, v_{i+2}v_{i+3}\}, T_{j})\leq8$ for every $1\leq i\leq l$ and $1\leq j\leq k-1$. Otherwise it may be assumed without loss of generality that $e(\{v_{1}v_{2}, v_{3}v_{4}\}, T_{1})\geq9$. Then by Lemma \ref{l4}, $G[\{v_{1}, v_{2}, v_{3}, v_{4}\}\cup V(T_{1})]$ contains two vertex-disjoint triangles, together with $T_{2}, \ldots, T_{k-1}$ constituting $k$ vertex-disjoint triangles contained in $G$, a contradiction. By Claim \ref{C17}, we have $e(\{v_{i}v_{i+1}, v_{i+2}v_{i+3}\}, T)\geq2(4k-5)=8k-10$ for every $1\leq i\leq l$. This implies that $e(\{v_{i}v_{i+1}, v_{i+2}v_{i+3}\}, T_{j})\geq6$ for every $1\leq i\leq l$ and $1\leq j\leq k-1$. In the rest of the proof, let $V(T_{i})=\{t_{i_{1}}, t_{i_{2}}, t_{i_{3}}\}$ for all $T_{i}\subseteq T$.

Now, we will show that $e(v_{i}, T_{j})\leq2$ for every $1\leq i\leq l$ and $1\leq j\leq k-1$. Suppose the contrary. It may be assumed without loss of generality that $e(v_{1}, T_{1})=3$. Then $e(v_{i}v_{i+1}, T_{1})\leq3$ for every $2\leq i\leq l-1$; otherwise there exists a vertex in $V(T_{1})$, say $t_{1_{1}}$, such that $t_{1_{1}}v_{i}v_{i+1}t_{1_{1}}$ is a triangle, together with $v_{1}t_{1_{2}}t_{1_{3}}v_{1}, T_{2}, \ldots, T_{k-1}$ constituting $k$ vertex-disjoint triangles contained in $G$, a contradiction. Since $e(\{v_{2}v_{3}, v_{l-1}v_{l}\}, T_{1})\geq6$, we know that $e(v_{2}v_{3}, T_{1})=3$ and $e(v_{l-1}v_{l}, T_{1})=3$. Furthermore, we have $e(v_{2}, T_{1})\leq2$; otherwise $e(\{v_{1}v_{2}, v_{l-1}v_{l}\}, T_{1})\geq3+3+3=9$. Similarly, we have $e(v_{l}, T_{1})\leq2$; otherwise $e(\{v_{l}v_{1}, v_{2}v_{3}\}, T_{1})\geq9$. If $e(v_{2}, T_{1})=0$, then $e(v_{i}v_{i+1}, T_{1})=3$ for every $3\leq i\leq l-1$; otherwise $e(\{v_{1}v_{2}, v_{i}v_{i+1}\}, T_{1})\leq3+2=5$. Hence we know that $e(v_{3}, T_{1})=3$, which implies that $e(v_{4}, T_{1})=0$, and so on. Finally, we have $e(v_{l}, T_{1})=3$, a contradiction. If $e(v_{2}, T_{1})=1$, then $e(v_{3}, T_{1})=2$ by $e(v_{2}v_{3}, T_{1})=3$. By Lemma \ref{W1}, we see that $G[\{v_{1}, v_{2}, v_{3}\}\cup V(T_{1})]$ contains two vertex-disjoint triangles, together with $T_{2}, \ldots, T_{k-1}$ constituting $k$ vertex-disjoint triangles contained in $G$, a contradiction. So $e(v_{2}, T_{1})=2$. Similarly, we have $e(v_{l}, T_{1})=2$. Now if $l\geq7$, then by $e(\{v_{i}v_{i+1}, v_{i+2}v_{i+3}\}, T_{1})\geq6$ for every $2\leq i\leq l-1$, we know that $e(v_{i}v_{i+1}, T_{1})=3$ with $2\leq i\leq l-1$. By $e(v_{2}, T_{1})=2$, one can easily check that $e(v_{3}, T_{1})=1$, $e(v_{4}, T_{1})=2$, \ldots, $e(v_{l}, T_{1})=1$, contrary to $e(v_{l}, T_{1})=2$. So we have $l=5$. Since $e(v_{2}, T_{1})=2$, $e(v_{2}v_{3}, T_{1})=3$, $e(v_{5}, T_{1})=2$ and $e(v_{4}v_{5}, T_{1})=3$, we see that $e(v_{3}, T_{1})=1$ and $e(v_{4}, T_{1})=1$. By Lemma \ref{W1}, either $G[\{v_{5}, v_{1}, v_{2}\}\cup V(T_{1})]$ contains two vertex-disjoint triangles, or $N_{T_{1}}(v_{5})=N_{T_{1}}(v_{2})$. If $G[\{v_{5}, v_{1}, v_{2}\}\cup V(T_{1})]$ contains two vertex-disjoint triangles, say $T'$ and $T''$, then $T', T''$ and $T_{2}, \ldots, T_{k-1}$ constitute $k$ vertex-disjoint triangles contained in $G$, a contradiction. If $N_{T_{1}}(v_{5})=N_{T_{1}}(v_{2})$, then $G[V(C)\cup V(T_{1})]$ contains two vertex-disjoint triangles, together with $T_{2}, \ldots, T_{k-1}$ constituting $k$ vertex-disjoint triangles contained in $G$, a contradiction.

Since $C$ is the shortest odd cycle in $G-T$, each vertex in $G-T-C$ can be adjacent to at most $2$ vertices in the $C$; otherwise we can find a shorter odd cycle. Hence
$$\frac{l(n+k-2)}{2}\leq\sum\nolimits_{v\in V(C)}d_{G}(v)\leq l(2k-2)+2l+2(n-3k+3-l).$$
If $v(G-T-C)=0$, then
$$\frac{l(n+k-2)}{2}\leq\sum\nolimits_{v\in V(C)}d_{G}(v)\leq l(2k-2)+2l=2kl.$$
We assert that $n\leq3k+2$. Otherwise $\frac{l(n+k-2)}{2}>2kl$. Since $n=3k-3+l\geq3k+2$, we have $n=3k+2$. Hence all the equalities are attained in above. This implies that $e(v_{i}, T_{j})=2$ for every $1\leq i\leq l$ and $1\leq j\leq k-1$. One can easily check that $G[V(C)\cup V(T_{1})]$ contains two vertex-disjoint triangles, together with $T_{2}, \ldots, T_{k-1}$ constituting $k$ vertex-disjoint triangles contained in $G$, a contradiction. So we have $v(G-T-C)\geq1$. Now $n\geq3k-3+l+1\geq3k+3$. We assert that $l\geq7$. Otherwise $\frac{l(n+k-2)}{2}>l(2k-2)+2l+2(n-3k+3-l)$. Since $e(v_{i}, T_{j})\leq2$ for every $v_{i}\in V(C)$ and $T_{j}\subseteq T$, we know that $e(v_{i}, G-T-C)\geq\frac{n+k-2}{2}-2(k-1)-2=\frac{n-3k-2}{2}$ for each $v_{i}\in V(C)$. Note that $v(G-T-C)=n-3k+3-l\leq n-3k-4$. Thus there exists a vertex $u\in V(G-T-C)$ such that $uv_{1}v_{2}u$ is a triangle, which implies that $G$ contains a $kC_{3}$, a contradiction.
\end{proof}

Now, let $T$ be a graph in $\mathcal{T}$ such that $G-T$ contains a $2K_{2}$ and let $T_{1}, T_{2}, \ldots, T_{k-1}$ denote $k-1$ vertex-disjoint triangles contained in $T$. By Claim \ref{C110}, we know that $G-T$ is a bipartite graph. Let $(A, B)$ be a bipartition of $G-T$. Since $G-T$ contains a $2K_{2}$, we have $|A|\geq2$ and $|B|\geq2$. If there exist vertices $a'\in A$ and $b'\in B$ such that $d_{G-T}(a')\geq\frac{n-3k+3}{2}$ and $d_{G-T}(b')\geq\frac{n-3k+3}{2}$, then $G-T=K_{\frac{n-3k+3}{2}, \frac{n-3k+3}{2}}$. Otherwise, we have $d_{G-T}(a)\leq\frac{n-3k+3}{2}-1=\frac{n-3k+1}{2}$ for every $a\in A$ or $d_{G-T}(b)\leq\frac{n-3k+3}{2}-1=\frac{n-3k+1}{2}$ for every $b\in B$. Without loss of generality, assume that $d_{G-T}(a)\leq\frac{n-3k+1}{2}$ for all $a\in A$. This implies that for any $a\in A$,
$$e(a, T)\geq\left\lceil\frac{n+k-2}{2}-\frac{n-3k+1}{2}\right\rceil\geq2k-1>2(k-1).$$
For any $2K_{2}$ of $G-T$, denote by $M$, we have $e(M, T_{i})\leq8$ for every $1\leq i\leq k-1$. Otherwise by Lemma \ref{l4}, $G[M\cup V(T_{i})]$ contains two vertex-disjoint triangles, together with $\{T_{1}, \ldots, T_{k-1}\}\setminus\{T_{i}\}$ constituting $k$ vertex-disjoint triangles contained in $G$, a contradiction. By Claim \ref{C17}, we have $e(M, T)\geq2(4k-5)=8k-10$. This implies that $e(M, T_{i})\geq6$ for every $1\leq i\leq k-1$. We distinguish two cases.

\setcounter{case}{0}
\begin{case}
$|A|=2$.
\end{case}

Let $A=\{a_{1}, a_{2}\}$ and let $M'=\{a_{1}b_{1}, a_{2}b_{2}\}$ be a $2K_{2}$ in $G-T$. Then
$$2n+2k-4=\frac{4(n+k-2)}{2}\leq\sum\nolimits_{v\in V(M')}d_{G}(v)\leq8(k-1)+\frac{2(n-3k+1)}{2}+4=n+5k-3.$$
This implies that $n\leq3k+1$; otherwise $2n+2k-4>n+5k-3$. Recall that $n\geq3k+1$. We have $n=3k+1$, which implies that the equalities are attained in the above. Hence we see that $d_{G-T}(b_{1})=d_{G-T}(b_{2})=2$ and $d_{G-T}(a_{1})=d_{G-T}(a_{2})=\frac{n-3k+1}{2}=1\neq2$, a contradiction.

\begin{case}
$|A|\geq3$.
\end{case}

Let $a$ be an arbitrary vertex of $A$. Since $e(a, T)>2(k-1)$, it may be assumed without loss of generality that $e(a, T_{1})=3$. Let $V(T_{1})=\{x, y, z\}$. If $G-T-a$ contains a $2K_{2}$, say, $M'=\{a_{1}b_{1}, a_{2}b_{2}\}$, then $e(T_{1}, a_{i}b_{i})\leq3$ with $i=1, 2$; otherwise there exists a vertex in $V(T_{1})$, say $x$, such that $xa_{i}b_{i}x$ is a triangle, together with $ayza, T_{2}, \ldots, T_{k-1}$ constituting $k$ vertex-disjoint triangles contained in $G$, a contradiction. Hence we have $e(M', T_{1})\leq6$. This implies that $e(a, T_{i})\leq2$ for every $2\leq i\leq k-1$; otherwise $e(M', T)\leq6+6+8(k-3)=8k-12<8k-10$. Hence we have $d_{G-T}(a)\geq\frac{n+k-2}{2}-(3+2(k-2))=\frac{n-3k}{2}$. Furthermore, we see that
$$2n+2k-4=\frac{4(n+k-2)}{2}\leq\sum\nolimits_{v\in V(M')}d_{G}(v)\leq6+8(k-2)+(n-3k+1)+2|A|.$$
By $2n+2k-4\leq 6+8(k-2)+(n-3k+1)+2|A|$, we have $|A|\geq\frac{n-3k+5}{2}$. This implies that $|B|=n-3k+3-|A|\leq\frac{n-3k+1}{2}$. Note that $\frac{n-3k}{2}\leq d_{G-T}(a)\leq\frac{n-3k+1}{2}$. If $n-3k$ is odd, then $G-T=K_{\frac{n-3k+5}{2}, \frac{n-3k+1}{2}}$. If $n-3k$ is even, then $G-T=K_{\frac{n-3k+6}{2}, \frac{n-3k}{2}}$.

If $G-T$ contains a $3K_{2}$, then $n\geq3k-3+6=3k+3$ and for any $a\in A$, $G-T-a$ contains a $2K_{2}$. From the above analysis, now whether $n-3k$ is odd or even, we see that $G-T$ is a complete bipartite graph whose two parts are at least $2$ in size.

Suppose that $G-T$ contains no $3K_{2}$'s. Let $M''=\{a'b', a''b''\}$ be a $2K_{2}$ in $G-T$. Then $d_{G-T}(a)\geq\frac{n-3k}{2}$ for all $a\in A\setminus\{a', a''\}$ and $|B|\leq\frac{n-3k+1}{2}$. Since $G-T$ contains no $3K_{2}$'s, we know that $d_{G-T}(a)\leq2$ for each $a\in A\setminus\{a', a''\}$. Hence $\frac{n-3k}{2}\leq2$, which implies that $n\leq3k+4$. So $|B|\leq\frac{n-3k+1}{2}\leq2$. Recall that $|B|\geq2$. We have $|B|=2$, which implies that $n\geq3k+3$. So $|A|=n-3k+3-|B|\geq4$ and for all $a\in A\setminus\{a', a''\}$, $d_{G-T}(a)\geq\frac{n-3k}{2}\geq2$. Since $|B|=2$, we have $d_{G-T}(a)=2$, where $a\in A\setminus\{a', a''\}$. This implies that both $G-T-a'$ and $G-T-a''$ contain a $2K_{2}$. Hence $d_{G-T}(a')=d_{G-T}(a'')=2$. We see that $G-T$ is a complete bipartite graph whose two parts are at least $2$ in size.

The proof of Lemma \ref{111} is complete. {\hfill$\Box$}

\end{document}